\documentclass[smallextended,referee,envcountsect,]{svjour3}

\smartqed
\usepackage{comment}
\usepackage{booktabs,framed}
\usepackage{placeins}
\usepackage{algorithm}
\usepackage[misc]{ifsym}
\usepackage{algpseudocode}
\usepackage{subcaption}
\usepackage{array}    
\usepackage{booktabs} 
\usepackage{amsfonts}
\usepackage{graphicx}
\usepackage{tcolorbox,enumerate}
\usepackage{amssymb,latexsym,amsmath,enumitem,amscd,amsthm,graphicx,color}
\usepackage{hyperref}
\allowdisplaybreaks

\newtheorem{assumption}{Assumption}
\begin{document}

\title{Adaptive Variance-Reduced Stochastic Sequential (Quadratic Constrained) Quadratic Programming}

\author{Gopal Tiwari, Md Abu Talhamainuddin Ansary}

\institute{Gopal Tiwari and  Md Abu Talhamainuddin Ansary \at
Department of Mathematics\\
IIT Jodhpur, Jodhpur, India 342030\\
Email:\\
G. Tiwari: gopalpersonalid2003@gmail.com\\
M. A. T, Ansary: md.abutalha2009@gmail.com (\Letter)
 }


\maketitle
\begin{abstract}
In this paper, we propose an adaptive variance-reduced stochastic sequential (quadratic constrained) quadratic programming (AVR-SS(QC)QP) method for nonlinear constrained optimization. The framework integrates SPIDER-based variance reduction, stochastic quasi-Newton curvature information, merit-function globalization, and adaptive switching between stochastic SQP and SQCQP models. The switching mechanism is governed by a relative constraint-model curvature indicator that measures the second-order constraint information neglected by the linearized SQP model. A smoothed hysteresis rule activates the SQCQP model only when the relative constraint curvature is significant. An exact \(L_1\) merit function with Armijo backtracking globalizes the resulting directions, while a coupled stochastic Lyapunov analysis controls the interaction between gradient-estimation error and optimization progress. Under standard smoothness, sampling, regularity, and second-order conditions, we establish almost-sure consistency of the variance-reduced gradient estimator and global convergence of the KKT residual. We further show that, near regular KKT points, the curvature indicator vanishes and the method eventually selects the computationally cheaper SQP model. Numerical experiments on a sparse principal component analysis problem illustrate the practical behavior of the proposed framework.
\end{abstract}

\section{Introduction}\label{sec1}
Large-scale optimization problems increasingly involve data-dependent objectives together with deterministic nonlinear constraints. Such problems arise in statistical learning, empirical risk minimization, data-driven
optimization, and constrained dimensionality reduction, where the objective is commonly represented by an expectation or a finite sum over a large data set, while feasibility is imposed through nonlinear equalities and
inequalities \cite{BottouCurtisNocedal2018,CurtisRobinsonZhou2024}.
In this setting, reduction of the objective-gradient cost alone is not sufficient, since stochastic error, feasibility, multiplier estimation, local-model accuracy, and globalization must be handled simultaneously.
Sequential quadratic programming (SQP) provides a natural framework for this purpose through quadratic local models, linearized constraints, and primal--dual information, together with well-established globalization mechanisms \cite{NocedalWright2006,GillMurraySaunders2005}.

Considerable attention has therefore been given to stochastic variants of
SQP. Equality-constrained stochastic sequential quadratic optimization was
studied by \cite{BerahasCurtisRobinsonZhou2021}, with subsequent extensions
to rank-deficient constraint Jacobians \cite{BerahasCurtisONeillRobinson2024}, predictive variance reduction \cite{BerahasShiYiZhou2023}, fully stochastic trust-region formulations \cite{FangNaMahoneyKolar2024}, and high-probability complexity guarantees \cite{BerahasXieZhou2025}. Related stochastic SQP schemes based on
differentiable exact augmented Lagrangians were developed for equality
constraints \cite{NaAnitescuKolar2023Adaptive} and later extended to
nonlinear equality and inequality constraints \cite{NaAnitescuKolar2023Ineq}. More recently, stochastic sequential quadratic optimization has been considered for stochastic objectives with deterministic nonlinear equality and inequality constraints \cite{CurtisRobinsonZhou2024}. These developments show that much of the
structure of deterministic nonlinear programming can be retained while expensive objective derivatives are replaced by stochastic estimates.

The accuracy of such estimates remains important because gradient noise
propagates to the search direction, multiplier estimates, merit prediction,
and step acceptance. Variance-reduction techniques provide a natural means
of controlling this effect. Recursive estimators such as SARAH, SPIDER,
and PAGE reduce stochastic-gradient error without requiring a full gradient
evaluation at every iteration 
\cite{NguyenLiuScheinbergTakac2017,FangLiLinZhang2018,LiBaoZhangRichtarik2021}. In particular, the SPIDER estimator recursively tracks gradient differences and provides an effective mechanism for finite-sum and expectation-based nonconvex optimization \cite{FangLiLinZhang2018}. Predictive variance reduction has also been incorporated into stochastic SQP \cite{BerahasShiYiZhou2023}, motivating the use of recursive variance reduction in stochastic constrained optimization.

Variance reduction, however, addresses only the stochastic objective model. A separate source of approximation is introduced by the treatment of nonlinear constraints. In standard SQP, the constraints are linearized, so their second-order curvature is omitted from the local subproblem. When the neglected curvature is small, the resulting quadratic program is computationally attractive; when it is significant, the local feasible geometry may be represented less accurately. Sequential quadratically constrained quadratic programming (SQCQP) provides an alternative by retaining quadratic constraint models \cite{FukushimaLuoTseng2003,Solodov2004}. Greater local fidelity is thereby obtained, although a more involved quadratically constrained subproblem must be solved.
This distinction gives rise to a local model-selection problem that is separate from stochastic-gradient estimation. Existing stochastic SQP methods generally retain linearized constraint models throughout the
iterations \cite{BerahasCurtisRobinsonZhou2021,BerahasShiYiZhou2023,NaAnitescuKolar2023Ineq,CurtisRobinsonZhou2024} whereas classical SQCQP methods retain quadratic constraint information \cite{FukushimaLuoTseng2003,Solodov2004}. Persistent use of SQP may discard useful curvature information, while persistent use of SQCQP may introduce unnecessary computational effort when constraint linearization is already adequate. This motivates an adaptive mechanism in which the local model is selected according to the magnitude of the neglected constraint curvature.

To address this issue, we propose the 
\emph{Adaptive Variance-Reduced Stochastic Sequential (Quadratic Constrained) Quadratic Programming} (AVR-SS(QC)QP) method. SPIDER variance reduction, adaptive
mini-batching, positive-definite damped quasi-Newton models, adaptive SQP--SQCQP switching, and exact \(\ell_1\)-merit globalization are combined within a common framework. Model selection is governed by a
\emph{Relative Constraint-Model Curvature Indicator} (RCMCI). The stochastic SQP direction is used as a common probe, and the quadratic constraint-curvature contribution discarded by the SQP linearization is measured relative to a stabilized local scale. A smoothed two-threshold hysteresis rule is then employed so that SQP is selected for small relative curvature, SQCQP is activated when the curvature becomes significant, and the preceding model is retained in the intermediate regime.

Model selection and globalization are treated separately. Regardless of the selected local model, the effective direction is globalized through a common exact \(\ell_1\) merit function with Armijo backtracking \cite{powell1978convergence,NocedalWright2006}. When an SQCQP candidate does not provide the required merit descent, the already available SQP probe is used as a safeguard, while the penalty parameter is adjusted when
additional emphasis on feasibility is required. In this way, second-order constraint information can be exploited locally without requiring SQCQP to be selected at every iteration.

The convergence analysis is developed by first establishing consistency of the SPIDER estimator and then examining the local relationship between the SQP and SQCQP directions. Since distinct positive-definite curvature approximations are employed in the two models, local direction stability is obtained through a perturbation analysis of the SQCQP KKT system evaluated
at the SQP primal--dual pair. Merit descent, multiplier boundedness, penalty stabilization, and stochastic-gradient consistency are subsequently combined through a coupled stochastic Lyapunov argument. Under the stated
smoothness, sampling, regularity, and second-order assumptions, almost-sure convergence of the KKT residual to zero is established. In addition, vanishing of the SQP probe and the RCMCI in the regular local KKT regime leads to eventual selection of the less expensive SQP model.

The framework is finally specialized to constrained Sparse Principal Component Analysis (Sparse PCA), where sparsity is imposed to improve the interpretability of the principal components \cite{ZouHastieTibshirani2006,DAspremontElGhaouiJordanLanckriet2007,GuerraUrzolaVanDeunVeraSijtsma2021}. The first sparse component is obtained by minimizing the negative empirical variance subject to a spherical equality constraint and a smooth \(\ell_1\)-type inequality constraint. The resulting structure is well suited to the proposed method because the relevant constraint curvatures can be evaluated inexpensively.
AVR-SS(QC)QP is compared numerically with the stochastic SQP method StoSQP-CRZ \cite{CurtisRobinsonZhou2024}, the adaptive augmented Lagrangian method AALM-LO \cite{LiuONeill2026}, and a hard-constrained Adam implementation based on \cite{KingmaBa2015} and
\cite{MarquezNeilaSalzmannFua2017}.

The principal contributions are summarized as follows:
\begin{itemize}
    \item We develop AVR-SS(QC)QP, which integrates SPIDER variance reduction, adaptive mini-batching, safeguarded quasi-Newton curvature models, exact-merit globalization, and adaptive SQP--SQCQP model selection within a unified stochastic constrained optimization framework.

    \item We introduce the max-based RCMCI with smoothing and hysteresis, providing a curvature-aware mechanism for deciding when the quadratic constraint information of SQCQP is locally warranted.
    \item We establish local SQP--SQCQP direction stability and develop a coupled stochastic convergence analysis yielding bounded multiplier and penalty sequences, merit-function descent, and almost-sure convergence of the complete KKT residual. The analysis further proves eventual SQP selection near regular KKT points.
    \item We specialize the framework to constrained Sparse PCA and evaluate it against StoSQP-CRZ, AALM-LO, and hard-constrained Adam using common objective, feasibility, KKT, stochastic-oracle, and runtime diagnostics.
\end{itemize}
\par The outline of the paper is as follows. Some prerequisites and the stochastic constrained problem formulation are presented in Section~\ref{sec:problem_formulation}. Section~\ref{sec:methodology} develops the proposed AVR-SS(QC)QP framework, including the variance-reduced gradient estimator, quasi-Newton curvature models, local SQP and SQCQP subproblems, RCMCI-based adaptive switching, and merit-function globalization. Section~\ref{sec:convergence} establishes estimator consistency, local SQP--SQCQP direction stability, merit descent, coupled stochastic Lyapunov descent, global KKT convergence, and eventual SQP selection. Section~\ref{sec:numerical_experiments} specializes the framework to constrained Sparse PCA and presents the experimental setup, benchmark comparisons, and numerical results. Finally, the concluding section summarizes the main findings and discusses directions for future research.

\section{Preliminaries}\label{sec:problem_formulation}
Consider the finite-sum stochastic nonlinear programming problem
\begin{equation*}
\label{eq:general_problem}
\begin{aligned}
(P):~~\min_{x\in\mathbb{R}^{n}}\quad &
F(x):=\frac{1}{N}\sum_{r=1}^{N}F_r(x)\\\text{s.t.}\quad &h(x)=0,\\&c(x)\leq 0,
\end{aligned}\end{equation*}
where $h:\mathbb{R}^{n}\rightarrow\mathbb{R}^{m}$ and
$c:\mathbb{R}^{n}\rightarrow\mathbb{R}^{p}$ are component wise smooth functions. Here, $N$ denotes the total number of samples. The exact gradient of $F$ can be written as \(\nabla F(x)=\frac{1}{N}\sum_{r=1}^{N}\nabla F_r(x) \). Clearly computing exact gradient of $F$ is very much expensive if $N$ is very large. In that case stochastic/min-batch gradients are used to solve $(P)$ using descent methods. For a mini-batch $\mathcal{B}_k\subseteq\{1,\ldots,N\}$ of size
$|\mathcal{B}_k|=b_k$, the mini-batch gradient of $F$ is defined by \( \nabla F_{b_k}(x)=\frac{1}{b_k}\sum_{r\in\mathcal{B}_k}\nabla F_r(x).\)The mini-batch size is selected adaptively according to
$b_k=\min\left\{b_{\max},\left\lceil b_0(1+\gamma_b k)\right\rceil\right\},$
where $b_0$ is the initial mini-batch size, $b_{\max}$ is the prescribed
maximum mini-batch size for non-refresh iterations, and
$\gamma_b>0$ controls the batch-size growth, following the motivation
of adaptive batch-size strategies for variance-reduced methods
\cite{ji2020history}.

The variance-reduced stochastic objective-gradient estimator is
constructed using periodic full-gradient refreshes and mini-batch
recursive updates:
\begin{equation}
\label{eq:spider_estimator}
\boxed{g^k=\begin{cases}\nabla F(x^k),&k\bmod q=0,\\[2mm]\nabla F_{B_k}(x^k)-\nabla F_{B_k}(x^{k-1})+g^{k-1},&k\bmod q\neq0.\end{cases}}\end{equation}
Here, $q\geq1$ denotes the refresh period. At a refresh iteration,
the full finite-sum gradient is evaluated, whereas at a non-refresh
iteration the estimator is updated using a mini-batch gradient
difference.  The stochastic objective-gradient estimation error is defined by $e^k
=g^k-\nabla F(x^k).$  Consequently, the estimator error is reset to zero at
every refresh iteration.


The Lagrangian associated with $(P)$ is defined by
$$ \mathcal{L}(x,\lambda,\mu) =F(x)+\lambda^{\top}h(x)+\mu^{\top}c(x),$$
where $\lambda\in\mathbb{R}^{m}$ and $\mu\in\mathbb{R}^{p}_{+}$ denote the equality and inequality multipliers, respectively. Its primal gradient is
$$\nabla_x\mathcal{L}(x,\lambda,\mu)=\nabla F(x)+J_h(x)^{\top}\lambda+J_c(x)^{\top}\mu. $$
Here $J_h(x)=\left[\nabla h_1(x)^{\top},\nabla h_2(x)^{\top},\dots, \nabla h_m(x)^{\top}\right]^{\top}$ and $J_c(x)=\left[\nabla c_1(x)^{\top},\nabla c_2(x)^{\top},\dots, \nabla c_p(x)^{\top}\right]^{\top}$
 denote the Jacobian matrices of the equality and inequality constraint
mappings, respectively. The stochastic Lagrangian gradient is defined by
$$\nabla_x\mathcal{L}^k=g^k+J_h(x^k)^{\top}\lambda^k+J_c(x^k)^{\top}\mu^k. $$
One can observe that the stochastic and exact Lagrangian gradients differ by
the same objective-gradient estimation error:
\begin{equation}
\label{eq:lagrangian_gradient_error}
\nabla_x\mathcal{L}^k - \nabla_x\mathcal{L}(x^k,\lambda^k,\mu^k)=e^k.\end{equation}
Consequently \(g^k=\nabla F(x^k)\) holds at a refresh iteration and hence \( e^k=0\) holds at a refresh iteration. 
At non-refresh iterations, $e^k$ generally remains nonzero and is controlled by the variance-reduced SPIDER recursion developed in the convergence analysis.
\par A point $(x^\star,\lambda^\star,\mu^\star)$ is said to satisfy the
Karush--Kuhn--Tucker (KKT) conditions for $(P)$ if 
\begin{align*}
    \nabla_x\mathcal{L}(x^\star,\lambda^\star,\mu^\star)&=0\\h(x^\star)&=0\\c(x^\star)&\leq0\\
\mu^\star\geq0,~~\mu^\star\odot c(x^\star)&=0\end{align*}
Exact KKT residual, defined by \begin{equation}\label{eq:kkt_residual}
 R_{\mathrm{KKT}}(x,\lambda,\mu)=\begin{bmatrix}\nabla_x\mathcal{L}(x,\lambda,\mu)
\\[1mm]h(x)\\[1mm]c(x)^+\\[1mm]\mu\odot c(x)\end{bmatrix} \end{equation} collects the KKT conditions. Here, $c(x)^+=\max\{c_j(x),0\}$, $ j=1,\ldots,p$ denotes the component wise positive part. The multiplier vector of inequality constraints are restricted to
$\mu\in\mathbb{R}_+^p$. Hence, \(R_{\mathrm{KKT}}(x,\lambda,\mu)=0\) is equivalent to primal feasibility, dual feasibility, complementarity, and stationary.
\subsection{Local SQP and SQCQP models}
\label{subsec:local_models}

At iteration $k$, the stochastic SQP model combines a first-order approximation of the objective with a positive-definite quadratic approximation of the local Lagrangian curvature, while the constraints are linearized. Let $B^k\in\mathbb{R}^{n\times n}$ be a symmetric positive-definite quasi-Newton approximation to the Hessian of the Lagrangian. The stochastic SQP direction ($d_{\mathrm{SQP}}^k$) is obtained by solving the following quadratic subproblem:
\begin{equation}\label{eq:sqp_subproblem}
\boxed{\begin{aligned}\min_{d\in\mathbb{R}^{n}}\quad &(g^k)^{\top}d+\frac{1}{2}d^{\top}B^k d\\\text{s.t.}\quad &h(x^k)+J_h(x^k)d=0,\\&c(x^k)+J_c(x^k)d\leq0.
\end{aligned}}\end{equation}

For the stochastic SQCQP model, the same stochastic first-order objective information is retained, while a distinct symmetric positive-definite objective-curvature model \(G^k\in\mathbb{R}^{n\times n}\) is employed. In addition, the second-order curvature terms of the constraints are retained. Its direction ($d_{\mathrm{SQCQP}}^k$) is obtained by solving the following subproblem:
\begin{equation}\label{eq:sqcqp_subproblem}\boxed{\begin{aligned}\min_{d\in\mathbb{R}^{n}}\quad &(g^k)^{\top}d+\frac{1}{2}d^{\top}G^k d\\\text{s.t.}\quad &h_i(x^k)+
\nabla h_i(x^k)^{\top}d+\frac{1}{2}d^{\top}\nabla^2 h_i(x^k)d=0,~~i=1,2,\dots,m,\\
&c_j(x^k)+\nabla c_j(x^k)^{\top}d+\frac{1}{2}d^{\top}\nabla^2 c_j(x^k)d\leq0,~~j=1,2,\dots,p.\end{aligned}}\end{equation}

The two local models therefore share the same stochastic objective gradient approximation $g^k$, while using distinct quasi-Newton curvature matrices $B^k$ and $G^k$. Their distinction arises from both the curvature approximation and the treatment of the constraints: SQP uses the Lagrangian-Hessian approximation $B^k$ together with first-order constraint linearizations, whereas SQCQP uses the
objective-Hessian approximation $G^k$ and retains the second-order constraint-curvature terms. Consequently, $d_{\mathrm{SQP}}^k$ provides a computationally inexpensive probe of the local constraint geometry, while $d_{\mathrm{SQCQP}}^k$ is computed when the RCMCI-based switching criterion indicates that the neglected second-order constraint information may be significant. The curvature matrices \(B^k\) and \(G^k\) are updated separately. The precise update formulas are  specified later.
\section{AVR-SS(QC)QP method for constrained nonlinear programming problems}\label{sec:methodology}
In this section, we propose an adaptive variance-reduced stochastic sequential (quadratically constrained) quadratic programming (AVR-SS(QC)QP) method that combines variance-reduced stochastic gradient estimation, quasi-Newton curvature approximation, adaptive selection between stochastic SQP and stochastic SQCQP local models, and merit-function globalization. The central idea is to use the stochastic SQP direction as an inexpensive probe of the local constraint geometry and to activate the more expensive stochastic SQCQP model only when the neglected second-order constraint information is sufficiently significant. The method therefore separates local model selection from global step acceptance: the RCMCI determines the appropriate local model, whereas the merit function and Armijo line search determine the accepted step length.

The adaptive mini-batch size used at non-refresh iterations is defined in the preliminaries. The SPIDER estimator $g^k$ and the corresponding mini-batch gradient $\nabla F_{b_k}(x)$ are likewise defined in \eqref{eq:spider_estimator} and the preceding preliminaries, respectively. At each refresh iteration, the exact finite-sum gradient is evaluated, whereas at each non-refresh iteration the estimator is updated using the recursive mini-batch gradient difference. Thus, the estimation error $e^k$ is reset to zero at every refresh iteration and is controlled between successive refreshes by the SPIDER error recursion established in the convergence analysis. The stochastic Lagrangian gradient $\nabla_x\mathcal{L}^k$, defined in the preliminaries, is used for constructing the stochastic Lagrangian-gradient displacement in the SQP curvature approximation, while $g^k$ provides the first-order objective information in both the local SQP and SQCQP models.

The local SQP and SQCQP models require positive-definite quasi-Newton approximations of the relevant curvature information. Rather than forming the exact Hessian explicitly at every iteration, AVR-SS(CQ)QP maintains two damped BFGS approximations, namely $B^k$ for the stochastic SQP model and $G^k$ for the stochastic SQCQP model.

For $k\geq1$, define the primal displacement by $s^k=x^k-x^{k-1}.$ For the stochastic SQP model, define the stochastic Lagrangian-gradient displacement by
\begin{equation}\label{eq:method_yk}y^{\mathcal{L},k}=\nabla_x\mathcal{L}^k-\nabla_x\mathcal{L}^{k-1}.\end{equation}
while the stochastic objective-gradient displacement is defined by
$y^{F,k}=g^k-g^{k-1}.$

Let \(\operatorname{DBFGS}_{\delta}(B,s,y)\) denote the safeguarded Powell-damped BFGS update associated with the curvature pair \((s,y)\), where \(0<\delta<1\) is the damping parameter. In addition to Powell damping \cite{powell1978convergence,NocedalWright2006}, the update is stabilized by rejecting numerically unreliable secant pairs \cite{kanzow2023regularization}.For the stochastic SQP curvature model, the Lagrangian secant pair \((s^k,y^{\mathcal{L},k})\) is employed, and the update is given by
\begin{equation}\label{eq:B_update}B^k=\operatorname{DBFGS}_{\delta_B}\left(B^{k-1},s^k,y^{\mathcal{L},k}\right),\qquad k\geq1.\end{equation}
Thus, \(B^k\) is maintained as a positive-definite approximation of the local Lagrangian curvature used in the stochastic SQP subproblem.

For the stochastic SQCQP model, the second-order curvature of the constraints is represented explicitly in the quadratic constraint models; hence \(G^k\) approximates only the objective curvature. Since the objective may be nonconvex, the secant pair \((s^k,y^{F,k})\) need not satisfy the BFGS curvature condition. Accordingly, the objective-curvature update is applied only when
$(s^k)^\top y^{F,k}>0,$ which preserves positive definiteness of the BFGS update \cite{NocedalWright2006}.

Accordingly, the safeguarded update is defined by
\begin{equation}\label{eq:G_update}G^k=\begin{cases}\operatorname{DBFGS}_{\delta_G}
\left(G^{k-1},s^k,y^{F,k}\right),&(s^k)^{\top}y^{F,k}>0,\\[2mm]G^{k-1},&(s^k)^{\top}y^{F,k}\leq0,\end{cases}\qquad k\geq1.\end{equation}
This curvature qualification prevents systematic negative objective curvature from being converted artificially into positive BFGS curvature. The positive-definite SQCQP objective model is instead retained whenever the objective secant pair fails the positive-curvature test.

The SQP model uses $B^k$ together with first-order constraint linearizations, whereas the SQCQP model uses $G^k$ together with the explicit second-order constraint-curvature terms. This separation avoids double counting of the constraint curvature in the SQCQP model.
Throughout the convergence analysis, the resulting quasi-Newton matrices satisfy the uniform spectral bounds
\begin{equation}\label{eq:method_B_G_bounds}
\boxed{\underline b I \preceq B^k \preceq \overline b I,\qquad \underline g I
\preceq G^k \preceq \overline g I, \qquad 0<\underline b\leq\overline b<\infty,
\qquad 0<\underline g\leq\overline g<\infty.}\end{equation}
The lower spectral bounds guarantee uniform positive curvature in the respective quadratic objective models, while the upper bounds prevent the quasi-Newton approximations from becoming arbitrarily large.

\subsection{RCMCI-based adaptive switching}
\label{subsec:method_rcmci}
The central component of AVR-SS(QC)QP is an adaptive mechanism that determines whether the locally cheaper stochastic SQP model provides an adequate approximation of the constraint geometry or whether the second-order constraint information retained by the stochastic SQCQP model is required. This decision is based on the Relative Constraint-Model Curvature Indicator (RCMCI). The stochastic SQP model uses first-order constraint linearizations. Consequently, for the $i$th equality constraint and the $j$th inequality constraint, the second-order terms
\[ \frac{1}{2}d^{\top}\nabla^2 h_i(x^k)d \quad\text{and}\quad \frac{1}{2}d^{\top}\nabla^2 c_j(x^k)d \]
represent the local curvature information neglected by the SQP model. The RCMCI measures the magnitude of these neglected terms relative to a constraint-scaled reference along the current stochastic SQP direction.
At every iteration, the stochastic SQP subproblem \eqref{eq:sqp_subproblem} is first solved, and its solution \(d_{\mathrm{SQP}}^k\) is used as the common probe direction.This probe is computed irrespective of whether SQP or SQCQP is subsequently selected by the RCMCI. Hence, model selection is based on a common, computationally inexpensive reference direction. For each equality constraint, define the neglected second-order constraint correction along the SQP probe direction by
\[q_{h_i,k}=\frac{1}{2}\left(d_{\mathrm{SQP}}^k\right)^{\top}\nabla^2h_i(x^k)
d_{\mathrm{SQP}}^k,\qquad i=1,\ldots,m,\]
and, similarly, for each inequality constraint define
\[q_{c_j,k}=\frac{1}{2}\left(d_{\mathrm{SQP}}^k\right)^{\top}\nabla^2c_j(x^k)
d_{\mathrm{SQP}}^k,\qquad j=1,\ldots,p.\]
The magnitude of each neglected curvature term is normalized by a constraint-scaled reference. For the $i$th equality constraint, define
\begin{equation}\label{eq:method_s_hi}
s_{h_i,k}=\max\left\{\left|h_i(x^k)+\nabla h_i(x^k)^{\top}d_{\mathrm{SQP}}^k \right|,\;\gamma_d\|\nabla h_i(x^k)\|_2\|d_{\mathrm{SQP}}^k\|_2,\;\epsilon_s \right\},\end{equation}
and, for the $j$th inequality constraint, define
\begin{equation}\label{eq:method_s_cj}
s_{c_j,k}=\max\left\{\left|c_j(x^k)+\nabla c_j(x^k)^{\top}d_{\mathrm{SQP}}^k \right|,\;\gamma_d\|\nabla c_j(x^k)\|_2\|d_{\mathrm{SQP}}^k\|_2,\;\epsilon_s
\right\}.\end{equation}

The first term in each scale represents the magnitude of the constraint prediction produced by the first-order SQP model. The second term provides a directional scaling based on the corresponding constraint gradient and the probe direction, while $\epsilon_s>0$ prevents the normalization from becoming arbitrarily small when the linearized constraint prediction is zero or numerically negligible. The directional term is motivated by the bound $\left|\nabla r(x^k)^{\top}d\right|\leq\|\nabla r(x^k)\|_2\|d\|_2$ for any differentiable scalar constraint function \(r\).

The normalized relative curvature ratios are defined by
\[\rho_{h_i,k}=\frac{|q_{h_i,k}|}{s_{h_i,k}},\qquad i=1,\ldots,m,\ \ \ and  \   \ \ \rho_{c_j,k}=\frac{|q_{c_j,k}|}{s_{c_j,k}},\qquad j=1,\ldots,p.\]

The raw RCMCI is then defined as
\begin{equation}\label{eq:method_rcmci}
\boxed{\eta_k=\max\left\{\max_{1\leq i\leq m}\rho_{h_i,k},\;\max_{1\leq j\leq p}\rho_{c_j,k}\right\}.}\end{equation}

The maximum is taken over all existing constraints. Thus, the presence of a substantial relative second-order correction in even one constraint is sufficient to indicate that the first-order SQP constraint model may be locally inadequate. Consequently, \(\eta_k\) provides a conservative measure of the largest neglected constraint-curvature effect.
Because the SQP probe direction depends on the stochastic estimator $g^k$, the raw indicator $\eta_k$ may exhibit iteration-to-iteration fluctuations. To reduce sensitivity to such stochastic fluctuations, we use the exponentially smoothed indicator  
\begin{equation}\label{eq:method_smoothed_rcmci}
\boxed{\bar{\eta}_k=\beta_\eta\bar{\eta}_{k-1}+(1-\beta_\eta)\eta_k,\qquad 0\leq\beta_\eta<1.}\end{equation}
and the smoothing recursion is initialized by $\bar{\eta}_{-1}=0.$
Two switching thresholds are selected such that $0<\eta_{\mathrm{low}}<\eta_{\mathrm{high}},$ and the initial model is set as $\mathcal{M}_{-1}=\mathrm{SQP}.$

The RCMCI-based model-selection rule is defined through the following hysteresis mechanism:
\begin{equation}\label{eq:method_switching_rule}
\boxed{\mathcal{M}_k=\begin{cases}\mathrm{SQP},&\bar{\eta}_k\leq\eta_{\mathrm{low}},\\[1mm]\mathrm{SQCQP},&\bar{\eta}_k\geq\eta_{\mathrm{high}},\\[1mm]\mathcal{M}_{k-1},&\eta_{\mathrm{low}}<\bar{\eta}_k<\eta_{\mathrm{high}}.\end{cases}} \end{equation}
When the relative neglected constraint curvature is sufficiently small, the stochastic SQP model is selected. When the indicator exceeds the upper threshold, the stochastic SQCQP model is selected so that the second-order constraint-curvature information is retained. When the smoothed indicator lies strictly between the two thresholds, the model selected at the preceding iteration is retained.This hysteresis mechanism prevents rapid alternation between SQP and SQCQP caused by stochastic fluctuations near the switching thresholds.
For the selected model, the final search direction is therefore
\begin{equation}\label{eq:method_selected_direction}
\widetilde d^k=\begin{cases}d_{\mathrm{SQP}}^k,&\mathcal{M}_k=\mathrm{SQP},\\[1mm]
d_{\mathrm{SQCQP}}^k,& \mathcal{M}_k=\mathrm{SQCQP}.\end{cases}\end{equation}
The direction \(\widetilde d^k\) represents the local-model candidate before merit-function globalization. If the SQCQP candidate fails the merit-descent safeguard introduced in the subsequent globalization subsection, the already computed SQP probe direction is used instead.The effective search direction \(d^k\) passed to the Armijo line search may therefore differ from \(\widetilde d^k\) only through this globalization safeguard.

The RCMCI therefore separates model selection from step acceptance: $\bar{\eta}_k$ determines whether the local SQP or SQCQP model is appropriate, while the merit-function globalization and Armijo backtracking determine the accepted step length. In this way, the method activates the more expensive quadratic-constraint model only when the estimated contribution of the neglected constraint curvature is sufficiently significant.
\subsection{Direction Safeguard and Merit-Function Globalization}
\label{subsec:method_globalization}
The RCMCI mechanism determines whether the SQP or SQCQP local model is preferred at iteration \(k\). Although the two subproblems employ different curvature representations, their directions remain locally comparable near a regular solution; this stability property is established in Lemma~\ref{lem:4.2}. The direction ultimately passed to globalization is denoted by \(d^k\).To globalize the local step, AVR-SS(QC)QP employs the exact \(\ell_1\) merit function
\begin{equation}\label{eq:method_merit}
\Phi(x;\rho)=F(x)+\rho V(x),\qquad V(x)=\|h(x)\|_1+\|c(x)^+\|_1,\end{equation}
where \(c(x)^+\) denotes the componentwise positive part. For the working direction \(d^k\), define the stochastic merit directional quantity
\begin{equation} \label{eq:method_merit_direction}
D_k^\Phi=(g^k)^\top d^k+\rho_k DV(x^k)[d^k].\end{equation}

If an RCMCI-selected SQCQP direction satisfies \(D_k^\Phi<0\), it is passed directly to the line search. If \(D_k^\Phi\geq0\) but \(DV(x^k)[d_{\mathrm{SQCQP}}^k]<0\), the direction is kept fixed and the penalty is increased geometrically,$\rho_k\leftarrow\gamma_\rho\rho_k,\ \gamma_\rho>1,$
until \(D_k^\Phi<0\). When \(DV(x^k)[d_{\mathrm{SQCQP}}^k]\geq0\), increasing the penalty cannot enforce merit descent; hence the already-computed SQP probe and its local multipliers are used as the globalization fallback. The finite construction of such an effective merit-descent direction is proved in Lemma~\ref{lem:4.4}.

Starting from \(\alpha=1\), Armijo backtracking generates \(\alpha=\beta^\ell\), \(\ell=0,1,\ldots\), with \(0<\beta<1\), and accepts the first step satisfying
\begin{equation}\label{eq:method_armijo}
\Phi(x^k+\alpha d^k;\rho_k)\leq\Phi(x^k;\rho_k)+\sigma\alpha D_k^\Phi,\qquad 0<\sigma<1.\end{equation}
The iterate is then updated as $x^{k+1}=x^k+\alpha_k d^k.$ The effect of the stochastic gradient error and finite termination of the line search are established in Lemmas~\ref{lem:4.3} and~\ref{lem:4.5}.

Let \((\widehat{\lambda}^k,\widehat{\mu}^k)\) denote the local multiplier pair associated with the effective direction. After an accepted regular step,
\[\lambda^{k+1}=\widehat{\lambda}^k,\qquad \mu^{k+1}=\widehat{\mu}^k.\]
With $M_{k+1}=\max\left\{1,\|\lambda^{k+1}\|_\infty,\|\mu^{k+1}\|_\infty\right\},$
the post-step penalty update is\begin{equation}\label{eq:method_penalty_update}
\rho_{k+1}=\begin{cases}\rho_k,&\rho_k\geq\gamma_\rho M_{k+1},\\[1mm] \gamma_\rho\rho_k,&\rho_k<\gamma_\rho M_{k+1}.\end{cases}\end{equation}
The boundedness of the multiplier estimates and finite stabilization of \(\rho_k\) are established in Section~\ref{subsec:merit_globalization}. For a numerically zero effective direction, convergence is declared only after verification of the exact KKT residual. All Armijo merit evaluations use the exact finite-sum objective and nonlinear constraint mappings; stochasticity enters the globalization mechanism through the estimator \(g^k\).

\begin{algorithm}[!ht]
\small
\caption{Adaptive Variance-Reduced Stochastic Sequential
(Quadratic Constrained) Quadratic Programming (AVR-SS(QC)QP)}
\label{alg:avr_sqcp}
\begin{algorithmic}[1]

\Require
$x^0,\lambda^0,\mu^0,B^0\succ0,G^0\succ0,\rho_0>0$;
$b_0,b_{\max},\gamma_b,q$;
$\eta_{\mathrm{low}},\eta_{\mathrm{high}},\beta_\eta,\gamma_d,\epsilon_s$;
$\delta_B,\delta_G,\sigma,\beta,\gamma_\rho$;
$\varepsilon_{\mathrm{SQCQP}},\varepsilon_{\mathrm{KKT}}$.

\State
$\mathcal{M}_{-1}\gets\mathrm{SQP}$,
$\bar{\eta}_{-1}\gets0$.

\For{$k=0,1,2,\ldots$}

    \State Compute the SPIDER estimator $g^k$ from
    \eqref{eq:spider_estimator} using the adaptive batch size $b_k$.

    \State
    $\nabla_x\mathcal{L}^k
    \gets
    g^k+J_h(x^k)^\top\lambda^k+J_c(x^k)^\top\mu^k$.

    \If{$k\geq1$}
        \State
        $s^k\gets x^k-x^{k-1}$,
        $y^{\mathcal{L},k}\gets
        \nabla_x\mathcal{L}^k-\nabla_x\mathcal{L}^{k-1}$,
        $y^{F,k}\gets g^k-g^{k-1}$.
        \State
        $B^k\gets
        \operatorname{DBFGS}_{\delta_B}
        (B^{k-1},s^k,y^{\mathcal{L},k})$.
        \State
        $G^k\gets
        \operatorname{DBFGS}_{\delta_G}
        (G^{k-1},s^k,y^{F,k})$
        if $(s^k)^\top y^{F,k}>0$;
        otherwise $G^k\gets G^{k-1}$.
    \EndIf

    \State Solve \eqref{eq:sqp_subproblem} to obtain
    $d_{\mathrm{SQP}}^k$ and
    $(\widehat{\lambda}_{\mathrm{SQP}}^k,
    \widehat{\mu}_{\mathrm{SQP}}^k)$.

    \State Compute $\eta_k$ from \eqref{eq:method_rcmci} and
    $\bar{\eta}_k$ from \eqref{eq:method_smoothed_rcmci};
    determine $\mathcal{M}_k$ using
    \eqref{eq:method_switching_rule}.

    \If{$\mathcal{M}_k=\mathrm{SQP}$}
        \State
        $(d^k,\widehat{\lambda}^k,\widehat{\mu}^k)
        \gets
        (d_{\mathrm{SQP}}^k,
        \widehat{\lambda}_{\mathrm{SQP}}^k,
        \widehat{\mu}_{\mathrm{SQP}}^k)$.
    \Else
        \State Solve \eqref{eq:sqcqp_subproblem}, initialized by
        $d_{\mathrm{SQP}}^k$, to tolerance
        $\varepsilon_{\mathrm{SQCQP}}$.
        \State Obtain
        $(d^k,\widehat{\lambda}^k,\widehat{\mu}^k)
        \gets
        (d_{\mathrm{SQCQP}}^k,
        \widehat{\lambda}_{\mathrm{SQCQP}}^k,
        \widehat{\mu}_{\mathrm{SQCQP}}^k)$.
    \EndIf

    \State Compute
    $D_k^\Phi=(g^k)^\top d^k+\rho_kDV(x^k)[d^k]$.

    \If{$\mathcal{M}_k=\mathrm{SQCQP}$,
    $D_k^\Phi\geq0$, and $DV(x^k)[d^k]\geq0$}
        \State Replace
        $(d^k,\widehat{\lambda}^k,\widehat{\mu}^k)$
        by the SQP probe quantities and recompute $D_k^\Phi$.
    \EndIf

    \If{$d^k$ is numerically zero}
        \State Check the exact KKT residual using both the current
        and selected local multiplier pairs; terminate if either passes
        $\varepsilon_{\mathrm{KKT}}$.
    \EndIf

    \While{$D_k^\Phi\geq0$}
        \If{$DV(x^k)[d^k]\geq0$}
            \State Declare globalization failure.
        \Else
            \State
            $\rho_k\gets\gamma_\rho\rho_k$ and recompute $D_k^\Phi$.
        \EndIf
    \EndWhile

    \State Determine $\alpha_k$ by Armijo backtracking using
    \eqref{eq:method_armijo}.

    \State
    $x^{k+1}\gets x^k+\alpha_kd^k$,
    $\lambda^{k+1}\gets\widehat{\lambda}^k$,
    $\mu^{k+1}\gets\widehat{\mu}^k$.

    \State
    $M_{k+1}\gets
    \max\{1,\|\lambda^{k+1}\|_\infty,
    \|\mu^{k+1}\|_\infty\}$ and update $\rho_{k+1}$
    using \eqref{eq:method_penalty_update}.

    \If{$\|R_{\mathrm{KKT}}
    (x^{k+1},\lambda^{k+1},\mu^{k+1})\|_2
    \leq\varepsilon_{\mathrm{KKT}}$}
        \State \Return
        $(x^{k+1},\lambda^{k+1},\mu^{k+1})$.
    \EndIf

\EndFor

\end{algorithmic}
\end{algorithm}

\section{Convergence Analysis}\label{sec:convergence}
This section establishes the convergence properties of AVR-SS(QC)QP. The analysis proceeds through the consistency of the variance-reduced gradient estimator, local stability of the SQP--SQCQP models, merit-function globalization, a coupled stochastic Lyapunov argument, and finally the KKT convergence of accumulation points. The asymptotic behavior of the RCMCI switching mechanism is established as a consequence of the resulting local KKT regime.

We first state the structural assumptions used throughout the convergence analysis. Properties such as boundedness of the multiplier sequence and finite stabilization of the penalty parameter are not assumed; they are established subsequently in Lemma~\ref{lem:4.6} and Theorem~\ref{thm:4.1}, respectively.

\begin{assumption}[Smoothness and bounded constraint curvature]\label{ass:4.1}
There exists a neighborhood $\mathcal{N}\subset\mathbb{R}^{n}$ containing
the relevant iterates such that, for every $r=1,\ldots,N$,
 \ $F_r$ is twice continuously differentiable on $\mathcal{N}$ and its gradient is $L_F$-Lipschitz:
\begin{equation}\label{eq:gradient_lipschitz}
\|\nabla F_r(x)-\nabla F_r(y)\|_2 \leq L_F\|x-y\|_2, \qquad x,y\in\mathcal{N}. \end{equation}
Moreover, $h$ and $c$ are twice continuously differentiable on $\mathcal{N}$, and there exist finite constants $M_h,M_c>0$ such that
\[ \|\nabla^2 h_i(x)\|_2\leq M_h, \qquad i=1,\ldots,m, \]
and \[ \|\nabla^2 c_j(x)\|_2\leq M_c, \qquad j=1,\ldots,p, \] for all $x\in\mathcal{N}$.
\end{assumption}
\begin{assumption}[Mini-batch sampling]\label{ass:4.2}
At every non-refresh iteration $k$, conditional on the filtration $\mathcal{F}_{k-1}$, the mini-batch
$\mathcal{B}_k\subseteq\{1,\ldots,N\}$ of cardinality
$|\mathcal{B}_k|=b_k$ is sampled independently and uniformly. Consequently, for every $x$ in the relevant neighborhood,
\begin{equation}\label{eq:unbiased_minibatch}
\mathbb{E}\left[\nabla F_{b_k}(x)\mid\mathcal{F}_{k-1}\right]=\nabla F(x), \end{equation}
When \(b_k=N\), the complete finite-sum index set \(\{1,\ldots,N\}\) is used.
\end{assumption}
\begin{assumption}[Uniform spectral bounds of the quasi-Newton matrices]\label{ass:4.3}
There exist constants $0<\underline b\leq\overline b<\infty$ and
$0<\underline g\leq\overline g<\infty$ such that, for all sufficiently
large $k$,
\begin{equation}\label{eq:BG_uniform_bounds}
\underline b I\preceq B^k\preceq\overline b I,\qquad \underline g I\preceq G^k\preceq\overline g I.\end{equation}
Thus, both local quadratic objective models remain uniformly
positive definite and bounded.
\end{assumption}

\begin{assumption}[Compactness of the primal iterates]\label{ass:4.4}
The primal sequence $\{x^k\}$ generated by AVR-SS(QC)QP remains in a
compact subset $\mathcal{X}\subset\mathcal{N}$.
\end{assumption}

\begin{assumption}[Local regularity and multiplier consistency]\label{ass:4.5}
At every accumulation point $x^\star$ of the primal sequence, the
gradients of the active equality and inequality constraints are
linearly independent. Moreover, in a neighborhood of each relevant
accumulation point, the active-set KKT systems associated with the
local SQP and SQCQP subproblems are strongly regular, with a uniformly bounded local inverse-Lipschitz modulus.
Let $(\widehat{\lambda}^k,\widehat{\mu}^k)$ denote the local multiplier pair associated with the effective direction $d^k$ passed to globalization. Along every relevant convergent subsequence,
\[\|\widehat{\lambda}^k-\lambda^k\|_2+\|\widehat{\mu}^k-\mu^k\|_2\longrightarrow0. \]

Furthermore, every corresponding accumulation point $(x^\star,\lambda^\star,\mu^\star)$ satisfies the second-order sufficient condition
\[w^\top\nabla_{xx}^2\mathcal{L}(x^\star,\lambda^\star,\mu^\star)w>0\]
for every nonzero direction $w$ in the associated critical cone.
\end{assumption}
\subsection{SPIDER Estimator Consistency}\label{subsec:spider_consistency}
The stochastic-gradient properties required in the subsequent globalization and KKT analysis are established first. Recall that $e^k=g^k-\nabla F(x^k)$ denotes the SPIDER estimation error. By construction of \eqref{eq:spider_estimator}, \(e^k=0\) at every refresh iteration.

\begin{lemma}[SPIDER error recursion and conditional variance]\label{lem:4.1}
Suppose Assumptions~\ref{ass:4.1} and~\ref{ass:4.2} hold. At every non-refresh iteration \(k\), $e^k=e^{k-1}+\xi^k,$ where
$\xi^k=\left[\nabla F_{\mathcal{B}_k}(x^k)-\nabla F_{\mathcal{B}_k}(x^{k-1})\right]-\left[\nabla F(x^k)-\nabla F(x^{k-1})\right].$
Moreover,
\[\mathbb{E}\left[\xi^k\mid\mathcal{F}_{k-1}\right]=0,and \
\left[\|\xi^k\|_2^2\mid\mathcal{F}_{k-1}\right]\leq\frac{L_F^2}{b_k}\|x^k-x^{k-1}\|_2^2.\]\end{lemma}
\begin{proof}
At a non-refresh iteration, $g^k=\nabla F_{\mathcal{B}_k}(x^k)-\nabla F_{\mathcal{B}_k}(x^{k-1})+g^{k-1}.$ Subtracting \(\nabla F(x^k)\), and adding and subtracting\(\nabla F(x^{k-1})\), gives\[e^k=e^{k-1}+\xi^k.\]
The conditional zero-mean property follows directly from the unbiased
mini-batch sampling in Assumption~\ref{ass:4.2}.
For the variance bound, define\[\Delta_r^k=\nabla F_r(x^k)-\nabla F_r(x^{k-1}).\]
The mini-batch gradient difference is the sample average of \(\Delta_r^k\), whereas its population mean is $\overline{\Delta}^k=\frac{1}{N}\sum_{r=1}^{N}\Delta_r^k=\nabla F(x_k)-\nabla F(x_{k-1}).$
Using the standard variance bound for uniform mini-batch sampling and
Assumption~\ref{ass:4.1},\[\|\Delta_r^k\|_2\leq L_F\|x^k-x^{k-1}\|_2,\]which yields
\[\mathbb{E}\left[\|\xi^k\|_2^2\mid\mathcal{F}_{k-1}\right]\leq\frac{L_F^2}{b_k}
\|x^k-x^{k-1}\|_2^2.\]\end{proof}

\begin{theorem}[Almost-sure consistency of the SPIDER estimator] \label{thm:spider_consistency}
Suppose Assumptions~\ref{ass:4.1} and~\ref{ass:4.2} hold and that
$\mathbb{E}\left[\sum_{k=1}^{\infty}\|x^k-x^{k-1}\|_2^2\right]<\infty.$
Then\[\|g^k-\nabla F(x^k)\|_2\longrightarrow0\qquad\text{a.s.}\]\end{theorem}
\begin{proof}
Let \(s(k)\) denote the most recent refresh iteration at or before \(k\). Since the refresh period \(q\) is finite by construction, $0\leq k-s(k)\leq q-1,$  and $e^{s(k)}=0.$ Repeated application of Lemma~\ref{lem:4.1} therefore gives
\[e^k=\sum_{j=s(k)+1}^{k}\xi^j.\]
Since the increments are conditionally centered, the cross terms vanish in expectation. Hence, by Lemma~\ref{lem:4.1},\[\mathbb{E}\|e^k\|_2^2\leq L_F^2 \sum_{j=s(k)+1}^{k}\mathbb{E}\|x^j-x^{j-1}\|_2^2.\]
Each iterate difference occurs in at most \(q\) such refresh windows,and therefore
\[\sum_{k=1}^{\infty}\mathbb{E}\|e^k\|_2^2\leq qL_F^2 \mathbb{E}\left[ \sum_{k=1}^{\infty} \|x^k-x^{k-1}\|_2^2\right]<\infty.\]
For every \(\varepsilon>0\), Markov's inequality gives
\[\sum_{k=1}^{\infty}\mathbb{P}\left(\|e^k\|_2>\varepsilon\right)<\infty.\]
The first Borel--Cantelli lemma then implies $\|e^k\|_2\longrightarrow0\ \text{a.s.},$ which proves the result.
\end{proof}

\begin{corollary}[Stochastic Lagrangian-gradient consistency]\label{cor:4.1}
Under the conditions of Theorem~\ref{thm:spider_consistency},
\[\left\|\nabla_x\mathcal{L}^k-\nabla_x\mathcal{L}(x^k,\lambda^k,\mu^k)\right\|_2
\longrightarrow0\qquad\text{a.s.}\]\end{corollary}
\begin{proof}
By the definitions of the stochastic and exact Lagrangian gradients,
\[\nabla_x\mathcal{L}^k-\nabla_x\mathcal{L}(x^k,\lambda^k,\mu^k)=g^k-\nabla F(x^k)
=e^k.\]The result follows immediately from Theorem~\ref{thm:spider_consistency}.
\end{proof}

\subsection{Local SQP--SQCQP Direction Stability}\label{subsec:direction_stability}

The stochastic SQP and SQCQP subproblems employ distinct curvature models.The following result shows that the two directions remain locally comparable near a regular solution.
Let \(\mathcal{A}_k\) denote the locally identified active inequality set and define $A^k=
\begin{bmatrix}
J_h(x^k)\ &
J_{c_{\mathcal{A}k}}(x^k)
\end{bmatrix}^{\top},\ r^k=\begin{bmatrix}h(x^k)\ &c_{\mathcal{A}_k}(x^k)\end{bmatrix}^{\top},$
and the corresponding local multiplier vector is written as
$\nu=\begin{bmatrix}\lambda\ &\mu_{\mathcal{A}_k}\end{bmatrix}^{\top}.$
Define the multiplier-weighted constraint-curvature matrix as \[H^k(\nu)=\sum_{i=1}^{m}\lambda_i\nabla^2h_i(x^k)+\sum_{j\in\mathcal{A}_k}\mu_j\nabla^2c_j(x^k),\]
let \(q^k(d)\) denote the vector of quadratic constraint corrections,
$q^k(d)=\begin{bmatrix}\frac12 d^\top\nabla^2h_i(x^k)d\\[1mm]\frac12 d^\top\nabla^2c_j(x^k)d\end{bmatrix}_{i=1,\ldots,m;\, j\in\mathcal{A}_k}
$ and $z_{\mathrm{SQP}}^k=\begin{bmatrix}d_{\mathrm{SQP}}^k\\\nu_{\mathrm{SQP}}^k
\end{bmatrix}$ denote a locally relevant KKT pair of the stochastic SQP subproblem.
On the identified active set, its KKT equations are
\begin{equation}\label{eq:sqp_kkt_system}
g^k + B^k d_{\mathrm{SQP}}^k + (A^k)^\top\nu_{\mathrm{SQP}}^k = 0, \qquad r^k + A^k d_{\mathrm{SQP}}^k = 0.\end{equation}
For the SQCQP model, define the local KKT mapping as
\[\Psi^k(d,\nu)=\begin{bmatrix}g^k+G^k d+(A^k)^\top\nu+H^k(\nu)d\\[1mm]r^k+A^k d+q^k(d)\end{bmatrix}.\]

\begin{lemma}[Local SQP--SQCQP direction stability]\label{lem:4.2}
Suppose Assumptions~\ref{ass:4.1}, \ref{ass:4.3}, and \ref{ass:4.5} hold. Consider a relevant convergent subsequence lying in the local regular regime of Assumption~\ref{ass:4.5}, and suppose that the SQP and SQCQP subproblems share the locally identified active set \(\mathcal{A}_k\). Then there exist constants
\(C_{\mathrm{dir}}>0\) and \(\delta_d>0\) such that, whenever $\|d_{\mathrm{SQP}}^k\|_2\leq\delta_d,$ the corresponding SQCQP direction satisfies
\[\boxed{\left\|d_{\mathrm{SQCQP}}^k-d_{\mathrm{SQP}}^k\right\|_2\leq C_{\mathrm{dir}}\left\|d_{\mathrm{SQP}}^k\right\|_2.}\]
Consequently,
$d_{\mathrm{SQP}}^k\longrightarrow0 \quad\Longrightarrow\quad d_{\mathrm{SQCQP}}^k\longrightarrow0.$\end{lemma}
\begin{proof}
The SQCQP KKT residual is first evaluated at the SQP KKT pair.Using the SQP stationarity and active-constraint equations, we get
\[\Psi^k\left(d_{\mathrm{SQP}}^k,\nu_{\mathrm{SQP}}^k\right)=\begin{bmatrix}
(G^k-B^k)d_{\mathrm{SQP}}^k+H^k(\nu_{\mathrm{SQP}}^k)d_{\mathrm{SQP}}^k\\[1mm]
q^k(d_{\mathrm{SQP}}^k)\end{bmatrix}.\]
By Assumption~\ref{ass:4.3}, $\|G^k-B^k\|_2\leq\overline g+\overline b.$
Along the convergent subsequence, the outer multipliers are bounded, and the multiplier consistency in Assumption~\ref{ass:4.5} therefore implies boundedness of the corresponding local multipliers which is proved later in \ref{lem:4.6}. Together with the bounded constraint Hessians in Assumption~\ref{ass:4.1}, there exists \(C_H>0\) such that \[\left\|H^k(\nu_{\mathrm{SQP}}^k)\right\|_2\leq C_H.\]
Similarly, bounded constraint curvature gives a constant \(C_q>0\) such that
\[\left\|q^k(d_{\mathrm{SQP}}^k)\right\|_2\leq C_q\left\|d_{\mathrm{SQP}}^k \right\|_2^2.\]
Hence, for suitable constants \(C_1,C_2>0\),
\[\left\|\Psi^k\left(d_{\mathrm{SQP}}^k,\nu_{\mathrm{SQP}}^k\right)\right\|_2\leq
C_1\left\|d_{\mathrm{SQP}}^k\right\|_2+C_2\left\|d_{\mathrm{SQP}}^k\right\|_2^2.\]
For sufficiently small \(\|d_{\mathrm{SQP}}^k\|_2\), the quadratic term is dominated by the linear term, so that
\[\left\|\Psi^k\left(d_{\mathrm{SQP}}^k,\nu_{\mathrm{SQP}}^k\right)\right\|_2\leq
C\left\|d_{\mathrm{SQP}}^k\right\|_2\] for some \(C>0\).

By the strong regularity stated in Assumption~\ref{ass:4.5}, the locally identified SQCQP KKT system possesses a locally Lipschitz inverse with uniformly bounded modulus. Therefore, for some \(C_{\mathrm{reg}}>0\),
\[\left\|z_{\mathrm{SQCQP}}^k-z_{\mathrm{SQP}}^k\right\|_2\leq C_{\mathrm{reg}}
\left\|\Psi^k\left(d_{\mathrm{SQP}}^k,\nu_{\mathrm{SQP}}^k\right)\right\|_2.\]
Combining the preceding bounds yields
\[\left\|z_{\mathrm{SQCQP}}^k-z_{\mathrm{SQP}}^k\right\|_2\leq C_{\mathrm{dir}}
\left\|d_{\mathrm{SQP}}^k\right\|_2\]
for some \(C_{\mathrm{dir}}>0\). Since the primal direction is a component of the KKT vector \(z\),
\[\left\|d_{\mathrm{SQCQP}}^k-d_{\mathrm{SQP}}^k\right\|_2\leq C_{\mathrm{dir}}
\left\|d_{\mathrm{SQP}}^k\right\|_2.\]
Finally, $\|d_{\mathrm{SQCQP}}^k\|_2\leq\left(1+C_{\mathrm{dir}}\right)
\|d_{\mathrm{SQP}}^k\|_2,$ which proves that \(d_{\mathrm{SQP}}^k\to0\) implies \(d_{\mathrm{SQCQP}}^k\to0\).\end{proof}

\subsection{Merit-Function Descent and Armijo Globalization}
\label{subsec:merit_globalization}

The effective direction \(d^k\), obtained after the SQCQP-to-SQP safeguard when necessary, is globalized using the exact nonlinear merit function and stochastic directional quantity defined in Eqs. \eqref{eq:method_merit} and \eqref{eq:method_merit_direction}
\begin{lemma}[Stochastic merit-function remainder]\label{lem:4.3}
Suppose Assumption~\ref{ass:4.1} holds. For every finite working penalty parameter \(\rho_k>0\), there exist \(\bar{\alpha}_k>0\) and \(C_\Phi^k>0\) such that, for every \(0<\alpha\leq\bar{\alpha}_k\),
\begin{equation}\label{eq:stochastic_merit_remainder}
\Phi(x^k+\alpha d^k;\rho_k)\leq\Phi(x^k;\rho_k)+\alpha D_k^\Phi+ \alpha\|e^k\|_2\|d^k\|_2+C_\Phi^k\alpha^2\|d^k\|_2^2,\end{equation} where $e^k=g^k-\nabla F(x^k),$ and one may take $C_\Phi^k=\frac{L_F}{2}+\frac{\rho_k}{2}\left(mM_h+pM_c\right).$\end{lemma}
\begin{proof}
By the Lipschitz continuity of \(\nabla F\) in Assumption~\ref{ass:4.1}, the descent lemma gives\[F(x^k+\alpha d^k)\leq F(x^k) + \alpha\nabla F(x^k)^\top d^k + \frac{L_F}{2}\alpha^2\|d^k\|_2^2.\]
Since $\nabla F(x^k)=g^k-e^k,$ Cauchy--Schwarz yields
\[F(x^k+\alpha d^k)\leq F(x^k)+\alpha(g^k)^\top d^k+\alpha\|e^k\|_2\|d^k\|_2
+\frac{L_F}{2}\alpha^2\|d^k\|_2^2.\]
It remains to control the constraint-violation term. Taylor's theorem and the Hessian bounds in Assumption~\ref{ass:4.1} imply, for each equality constraint,
\[|h_i(x^k+\alpha d^k)|\leq|h_i(x^k)|+\alpha D|h_i(x^k)|[d^k]+\frac{M_h}{2} \alpha^2\|d^k\|_2^2.\]
Likewise, using the one-sided directional derivative of the positive-part mapping,
\[c_j(x^k+\alpha d^k)^+\leq c_j(x^k)^+ +\alpha Dc_j(x^k)^+[d^k]+\frac{M_c}{2}
\alpha^2\|d^k\|_2^2.\]
Summation over all constraints gives
\[V(x^k+\alpha d^k)\leq V(x^k)+\alpha DV(x^k)[d^k]+\frac{mM_h+pM_c}{2} \alpha^2\|d^k\|_2^2.\]
Multiplying this inequality by \(\rho_k\), adding the objective bound, and using the definition of \(D_k^\Phi\) gives\eqref{eq:stochastic_merit_remainder}. \end{proof}

\begin{lemma}[Finite construction of an effective merit-descent direction] \label{lem:4.4}
Suppose Assumption~\ref{ass:4.3} holds and that the local SQP and SQCQP subproblems admit the relevant KKT solutions. At any nonstationary iteration for which the effective direction is nonzero, the AVR-SS(QC)QP globalization procedure obtains, after finitely many within-iteration penalty adjustments, $D_k^\Phi<0.$
\end{lemma}
\begin{proof}
Consider first an RCMCI-selected SQCQP candidate. If its stochastic merit directional quantity is already negative, no modification is required.
Suppose instead that $D_k^\Phi\geq0.$ If $DV(x^k)[d_{\mathrm{SQCQP}}^k]\geq0,$ increasing the penalty parameter cannot create descent. By the globalization safeguard, the SQCQP candidate is therefore replaced by the already-computed SQP probe and its associated local multipliers. If, on the other hand, $DV(x^k)[d_{\mathrm{SQCQP}}^k]<0,$ then, with the direction held fixed, $D_k^\Phi(\rho)=(g^k)^\top d_{\mathrm{SQCQP}}^k+\rho\,DV(x^k)[d_{\mathrm{SQCQP}}^k]$ is strictly decreasing in \(\rho\). Since                               
\(\rho\leftarrow\gamma_\rho\rho\) with \(\gamma_\rho>1\), a finite number of penalty increases produces \(D_k^\Phi<0\).

It remains to consider an effective SQP direction. Feasibility of the linearized SQP constraints implies $DV(x^k)[d_{\mathrm{SQP}}^k]\leq -V(x^k).$ Hence, if \(V(x^k)>0\), and $DV(x^k)[d_{\mathrm{SQP}}^k]<0,$ and the same geometric penalty increase produces \(D_k^\Phi<0\) after finitely many updates whenever it is not already negative.

If \(V(x^k)=0\), the current point is feasible. Let 
\((\widehat{\lambda}_{\mathrm{SQP}}^k, \widehat{\mu}_{\mathrm{SQP}}^k)\)
be the SQP subproblem multipliers. Taking the inner product of the SQP stationarity condition with \(d_{\mathrm{SQP}}^k\), and using the linearized constraint equations and complementarity, gives
\[(g^k)^\top d_{\mathrm{SQP}}^k\leq-(d_{\mathrm{SQP}}^k)^\top
B^kd_{\mathrm{SQP}}^k.\] Assumption~\ref{ass:4.3} therefore yields
\[(g^k)^\top d_{\mathrm{SQP}}^k\leq-\underline b \|d_{\mathrm{SQP}}^k\|_2^2<0\]
whenever \(d_{\mathrm{SQP}}^k\neq0\). Since \(DV(x^k)[d_{\mathrm{SQP}}^k]\leq0\), it follows that \(D_k^\Phi<0\).

Thus, every nonzero effective direction either already satisfies the stochastic merit-descent condition or attains it after finitely many penalty adjustments. A numerically zero direction is treated separately by the exact KKT safeguard described in Section~\ref{subsec:method_globalization}.
\end{proof}
\begin{lemma}[Armijo backtracking termination]\label{lem:4.5}
Suppose the conditions of Lemma~\ref{lem:4.3} hold and \(D_k^\Phi<0\). Assume further that, for some constant $0\leq\vartheta<1-\sigma,$ the estimator error satisfies \[\|e^k\|_2\|d^k\|_2\leq\vartheta(-D_k^\Phi).\]
Then the Armijo backtracking procedure terminates after finitely many reductions.
\end{lemma}
\begin{proof}
From Lemma~\ref{lem:4.3},\[\Phi(x^k+\alpha d^k;\rho_k)\leq\Phi(x^k;\rho_k)+\alpha D_k^\Phi+\alpha\|e^k\|_2\|d^k\|_2+C_\Phi^k\alpha^2\|d^k\|_2^2.\]
Using the assumed estimator-accuracy bound gives
\[\Phi(x^k+\alpha d^k;\rho_k)\leq\Phi(x^k;\rho_k)+\alpha(1-\vartheta)D_k^\Phi+
C_\Phi^k\alpha^2\|d^k\|_2^2.\]
Because \(D_k^\Phi<0\) and \(1-\sigma-\vartheta>0\), the Armijo condition
$\Phi(x^k+\alpha d^k;\rho_k)\leq\Phi(x^k;\rho_k)+\sigma\alpha D_k^\Phi$ is guaranteed whenever $C_\Phi^k\alpha\|d^k\|_2^2\leq(1-\sigma-\vartheta)(-D_k^\Phi).$
Thus, for \(d^k\neq0\), every sufficiently small positive step length satisfies the Armijo condition. Since the backtracking sequence 
\(\{1,\beta,\beta^2,\ldots\}\) tends to zero for \(0<\beta<1\), a finite accepted step \(\alpha_k>0\) exists.
\end{proof}
Consequently, whenever the estimator-accuracy condition of Lemma~\ref{lem:4.5} holds, the accepted step satisfies \[\Phi(x^{k+1};\rho_k)\leq\Phi(x^k;\rho_k)+
\sigma\alpha_kD_k^\Phi<\Phi(x^k;\rho_k).\]
The role of the stochastic estimator error in ensuring the asymptotic
validity of this condition is addressed in the subsequent coupled
Lyapunov analysis.
\begin{lemma}[Boundedness of multiplier estimates]\label{lem:4.6}
Suppose Assumptions~\ref{ass:4.1}--\ref{ass:4.5} hold. Then, on the regular tail associated with the accumulation set of \(\{x^k\}\), the local multiplier pairs generated by the effective SQP or SQCQP subproblems are uniformly bounded. Consequently, under the multiplier handoff
\[\lambda^{k+1}=\widehat{\lambda}^k,\qquad\mu^{k+1}=\widehat{\mu}^k,\]
there exist constants \(M_\lambda,M_\mu>0\) and a finite index \(K_\mu\) such that
\begin{equation}\label{eq:multiplier_boundedness}
\boxed{\|\lambda^k\|_\infty\leq M_\lambda,\qquad\|\mu^k\|_\infty\leq M_\mu,\qquad
k\geq K_\mu.}\end{equation}\end{lemma}
\begin{proof}
Let \(\Omega\) denote the set of accumulation points of\(\{x^k\}\). By Assumption~\ref{ass:4.4}, \(\{x^k\}\subset\mathcal X\) for a compact set \(\mathcal X\);hence \(\Omega\) is compact.

Since each \(F_r\) is continuously differentiable on \(\mathcal X\),there exists \(M_F>0\) such that $\|\nabla F_r(x)\|_2\leq M_F,\qquad x\in\mathcal X,\quad r=1,\ldots,N.$ At refresh iterations \(g^k=\nabla F(x^k)\), while between two
successive refreshes the SPIDER recursion contains at most \(q-1\) increments, each satisfying\[\left\|\nabla F_{\mathcal B_k}(x^k)-\nabla F_{\mathcal B_k}(x^{k-1})\right\|_2\leq2M_F.\]
Therefore, there exists \(M_g<\infty\) such that $\|g^k\|_2\leq M_g$ for all \(k\).

By Assumption~\ref{ass:4.1} and compactness of \(\mathcal X\), the constraint values, Jacobians, and Hessians are uniformly bounded, while Assumption~\ref{ass:4.3} gives uniform bounds on \(B^k\) and \(G^k\). Hence all data entering the local SQP and SQCQP KKT systems remain bounded.

For every \(x^\star\in\Omega\), Assumption~\ref{ass:4.5} provides a neighborhood \(\mathcal U(x^\star)\) in which the relevant local KKT system is strongly regular with bounded inverse-Lipschitz modulus. Consequently, the corresponding primal--dual KKT pair$z^k=\begin{bmatrix}d^k\\\widehat{\nu}^k\end{bmatrix}$ is bounded whenever \(x^k\in\mathcal U(x^\star)\).

Since \(\Omega\) is compact, finitely many such neighborhoods cover \(\Omega\). The sequence \(\{x^k\}\) eventually lies in their union; otherwise, a subsequence outside the union would possess, by compactness of \(\mathcal X\), an accumulation point in \(\Omega\), yielding a contradiction. Taking the maximum of the finitely many local bounds therefore gives \(M_z<\infty\) and \(K_0\) such that
\[\left\|\begin{bmatrix}d^k\\\widehat{\nu}^k\end{bmatrix}\right\|_2\leq M_z,\qquad
k\geq K_0.\]
Thus the local multiplier components \(\widehat{\lambda}^k\) and \(\widehat{\mu}^k\) are uniformly bounded on the regular tail.
Finally, the multiplier handoff$\lambda^{k+1}=\widehat{\lambda}^k,\ \mu^{k+1}=\widehat{\mu}^k$ implies that the outer multiplier sequence inherits this boundedness. Hence there exist finite constants \(M_\lambda,M_\mu>0\) and an index \(K_\mu\) such that \ $\|\lambda^k\|_\infty\leq M_\lambda,\\|\mu^k\|_\infty\leq M_\mu,\ k\geq K_\mu,$
which proves \eqref{eq:multiplier_boundedness}.
\end{proof}
\begin{theorem}[Finite stabilization of the penalty parameter]\label{thm:4.1}
Suppose Assumptions~\ref{ass:4.1}--\ref{ass:4.5} hold and let the multiplier sequences satisfy the boundedness conclusion of Lemma~\ref{lem:4.6}. Assume additionally that, on the relevant regular tail of the iterates, the merit-descent penalty thresholds are uniformly bounded: there exist a finite constant
\(\rho_{\mathrm{des}}>0\) and an index \(K_{\mathrm{des}}\) such that, for every \(k\geq K_{\mathrm{des}}\), no within-iteration penalty increase is required whenever the working penalty satisfies \[\rho_k\geq\rho_{\mathrm{des}}.\]

Then the penalty sequence generated by the within-iteration merit-descent adjustment and the post-step update \[\rho_{k+1}=\begin{cases}\rho_k,&\rho_k\geq
\gamma_\rho\max\left\{1,\|\lambda^{k+1}\|_\infty,\|\mu^{k+1}\|_\infty\right\},\\[1mm]\gamma_\rho\rho_k,&\text{otherwise},\end{cases}\qquad \gamma_\rho>1,\]
is bounded above and can increase only finitely many times.
Consequently, there exist a finite index \(K_\rho\) and a finite constant \(\bar\rho>0\) such that
\begin{equation}\label{eq:penalty_stabilization}
\boxed{\rho_k=\bar\rho,\qquad k\geq K_\rho.}\end{equation}\end{theorem}
\begin{proof}
By Lemma~\ref{lem:4.6}, there exist finite constants \(M_\lambda,M_\mu>0\) and an index \(K_\mu\) such that the multiplier sequences are bounded for all \(k\geq K_\mu\). Define $M_\star=\max\{1,M_\lambda,M_\mu\}.$Hence, for all sufficiently large \(k\),\[\max\left\{1,\|\lambda^{k+1}\|_\infty,\|\mu^{k+1}\|_\infty\right\}
\leq M_\star.\]
Therefore, once $\rho_k\geq\gamma_\rho M_\star,$ the post-step multiplier update cannot increase the penalty.
By hypothesis, on the same regular tail no within-iteration increase is required whenever $\rho_k\geq\rho_{\mathrm{des}}.$ If \(\rho_k<\rho_{\mathrm{des}}\), the geometric update \(\rho\leftarrow\gamma_\rho\rho\) can overshoot this threshold by at most one factor \(\gamma_\rho\), so the resulting working penalty is smaller than $\gamma_\rho\rho_{\mathrm{des}}.$ Moreover, an SQCQP candidate satisfying
\(D_k^\Phi\geq0\) and \(DV(x^k)[d_{\mathrm{SQCQP}}^k]\geq0\) does not trigger indefinite penalty growth, since the algorithm switches to the SQP probe.

Now suppose  $K_0=\max\{K_\mu,K_{\mathrm{des}}\},\qquad\rho_\star=\max\left\{\rho_{K_0}, \gamma_\rho\rho_{\mathrm{des}},\gamma_\rho^2M_\star\right\}.$
The within-iteration adjustment cannot exceed \(\gamma_\rho\rho_{\mathrm{des}}\), while a post-step increase from \(\rho_k<\gamma_\rho M_\star\) gives
$\rho_{k+1}=\gamma_\rho\rho_k<\gamma_\rho^2M_\star.$

Thus\[\rho_k\leq\rho_\star,\qquad k\geq K_0.\]
Finally, \(\rho_k\) never decreases, and whenever it increases it is multiplied by the fixed factor \(\gamma_\rho>1\). Hence infinitely many increases would imply \(\rho_k\to\infty\), contradicting the finite upper bound \(\rho_\star\). Therefore only finitely many penalty increases are possible. Consequently, there exist
\(K_\rho<\infty\) and \(\bar\rho>0\) such that\[\rho_k=\bar\rho,\qquad k\geq K_\rho,\] which proves \eqref{eq:penalty_stabilization}.\end{proof}
\subsection{Coupled Stochastic Lyapunov Analysis}\label{subsec:coupled_lyapunov}
The merit-function decrease and the SPIDER estimation error are not independent. Between two successive refresh iterations, the SPIDER error may increase as the primal iterate moves, whereas the accepted Armijo step decreases the stabilized merit function. We therefore analyze these two effects through a coupled Lyapunov function.Throughout this subsection, let \(\rho_k=\bar\rho\) for \(k\geq K_\rho\), as guaranteed by Theorem~\ref{thm:4.1}.
\begin{lemma}[Uniform merit descent on the stabilized tail]\label{lem:4.7}
Suppose Assumptions~\ref{ass:4.1}--\ref{ass:4.5} hold and
\(\rho_k=\bar\rho\) for all sufficiently large \(k\). Define
\[\widehat M=\frac{\bar\rho}{\gamma_\rho},\qquad\rho_{\mathrm{gap}}=\bar\rho
\widehat M=\bar\rho\left(1-\frac1{\gamma_\rho}\right)>0,\] \[ and \qquad C_q=\frac12(mM_h+pM_c),\ C_H=\widehat M(mM_h+pM_c),\ \kappa_Q=\underline g-\frac12C_H-\bar\rho C_q>0.\]
Let $\kappa_\Phi=\min\{\underline b,\kappa_Q\}>0.$ Then, for every sufficiently large regular iteration, the effective direction \(d^k\) satisfies
\begin{equation}\label{eq:uniform_merit_descent}
D_k^\Phi\leq -\kappa_\Phi\|d^k\|_2^2-\rho_{\mathrm{gap}}V(x^k).
\end{equation}\end{lemma}
\begin{proof}
Since the penalty has stabilized, the post-step update and multiplier handoff imply $\|\widehat\lambda^k\|_\infty,\|\widehat\mu^k\|_\infty \leq\widehat M$ for all sufficiently large \(k\).

If \(d^k=d_{\mathrm{SQP}}^k\), equations \eqref{eq:sqp_kkt_system} yield
\[(g^k)^\top d^k=-(d^k)^\top B^kd^k+(\widehat\nu^k)^\top r^k\leq-\underline b\|d^k\|_2^2+\widehat M V(x^k).\]
Moreover, SQP linearized feasibility gives \(DV(x^k)[d^k]\leq-V(x^k)\).

Hence $D_k^\Phi\leq-\underline b\|d^k\|_2^2-\rho_{\mathrm{gap}}V(x^k).$
If \(d^k=d_{\mathrm{SQCQP}}^k\), the SQCQP KKT conditions, 
\(r^k+A^kd^k+q^k(d^k)=0\), and$(\widehat\nu^k)^\top q^k(d^k)=\frac12(d^k)^\top
H^k(\widehat\nu^k)d^k$ give \[(g^k)^\top d^k\leq-\left(\underline g-\frac12C_H\right)\|d^k\|_2^2+\widehat M V(x^k),\]
where we used $\|H^k(\widehat\nu^k)\|_2\leq C_H,\qquad(\widehat\nu^k)^\top r^k\leq\widehat M V(x^k).$ The quadratic constraint model also implies
\[DV(x^k)[d^k]\leq-V(x^k)+C_q\|d^k\|_2^2.\]
Therefore,
\[D_k^\Phi\leq-\kappa_Q\|d^k\|_2^2-\rho_{\mathrm{gap}}V(x^k).\]
Taking \(\kappa_\Phi=\min\{\underline b,\kappa_Q\}\) establishes \eqref{eq:uniform_merit_descent}. The SQP bound also applies when the globalization safeguard replaces an SQCQP candidate by the SQP probe.
\end{proof}
\begin{lemma}[Estimator-error energy recursion]\label{lem:4.8}
Suppose Assumptions~\ref{ass:4.1} and~\ref{ass:4.2} hold. If \(k+1\) is a non-refresh iteration, then \[\mathbb E\left[E^{k+1}\mid\mathcal F_k\right]\leq
E^k+\frac{L_F^2}{b_{k+1}}\alpha_k^2\|d^k\|_2^2.\]
If \(k+1\) is a refresh iteration, then $E^{k+1}=0.$
\end{lemma}
\begin{proof}
At a non-refresh iteration, Lemma~\ref{lem:4.1} gives$e^{k+1}=e^k+\xi^{k+1},\
\mathbb E\left[\xi^{k+1}\mid\mathcal F_k\right]=0.$
Hence
\[\begin{aligned}\mathbb E\left[E^{k+1}\mid\mathcal F_k\right]&=\mathbb E\left[
\|e^k+\xi^{k+1}\|_2^2\mid\mathcal F_k\right]\\&=E^k+\mathbb E\left[ \|\xi^{k+1}\|_2^2\mid\mathcal F_k\right],\end{aligned}\] because the conditional cross term vanishes.

By Lemma~\ref{lem:4.1}, and $x^{k+1}-x^k=\alpha_kd^k,$ the claimed recursion follows.If \(k+1\) is a refresh iteration, the SPIDER definition gives
$g^{k+1}=\nabla F(x^{k+1}),$ and therefore \(E^{k+1}=0\).
\end{proof}

Let $r_k=k\bmod q$ denote the position of iteration \(k\) in the current SPIDER refresh cycle. Choose a constant \(\omega>0\) satisfying $q\omega\frac{L_F^2}{b_0}\leq\frac{\sigma\kappa_\Phi}{2},$ which is always possible because \(q<\infty\), \(b_0>0\), and \(\kappa_\Phi>0\). Define $C_r=\omega(q-r),\qquad r=0,\ldots,q-1,$ and introduce the coupled Lyapunov function
\begin{equation}\label{eq:coupled_lyapunov}
\Psi^k=\Phi(x^k;\bar\rho)+C_{r_k}E^k.\end{equation}

For every non-refresh transition,$C_{r_k}-C_{r_{k+1}}=\omega,$
whereas at a refresh transition the estimator energy is reset to zero.

\begin{theorem}[Coupled stochastic Lyapunov descent] \label{thm:4.2}
Suppose Assumptions~\ref{ass:4.1}--\ref{ass:4.5} hold, the conclusions of Theorem~\ref{thm:4.1} and Lemma~\ref{lem:4.7} hold, and the estimator-accuracy condition of Lemma~\ref{lem:4.5} is satisfied on the relevant tail so that the Armijo line search is well defined.

Then there exists a finite index \(K_L\) such that, for every
\(k\geq K_L\),
\begin{equation}\label{eq:coupled_lyapunov_descent}
\boxed{\mathbb E\left[\Psi^{k+1}\mid\mathcal F_k\right]\leq\Psi^k-a_1\alpha_k\|d^k\|_2^2-a_2E^k-a_3\alpha_kV(x^k),}
\end{equation}

where$ \ a_1=\frac{\sigma\kappa_\Phi}{2},\qquad a_2=\omega,\qquad a_3= \sigma\rho_{\mathrm{gap}}.$
\end{theorem}
\begin{proof}
We consider separately a non-refresh transition and a refresh transition.
\medskip
\noindent
\textbf{Case 1: \(k+1\) is a non-refresh iteration.}
By Lemma~\ref{lem:4.5}, the Armijo backtracking is well defined on the considered tail. Hence 
\[\Phi(x^{k+1};\bar\rho)\leq\Phi(x^k;\bar\rho)+\sigma\alpha_kD_k^\Phi.\]
Using Lemma~\ref{lem:4.7},
\[\Phi(x^{k+1};\bar\rho)\leq\Phi(x^k;\bar\rho)-\sigma\kappa_\Phi\alpha_k\|d^k\|_2^2
-\sigma\rho_{\mathrm{gap}}\alpha_kV(x^k).\]

At the same transition, Lemma~\ref{lem:4.8} gives
\[\mathbb E\left[E^{k+1}\mid\mathcal F_k\right]\leq E^k+\frac{L_F^2}{b_{k+1}}
\alpha_k^2\|d^k\|_2^2.\]
Multiplying by \(C_{r_{k+1}}\), adding the merit estimate, and using the definition of \(\Psi^k\) yields
\[\begin{aligned}\mathbb E\left[\Psi^{k+1}\mid\mathcal F_k\right]\leq{}&\Psi^k
-\left(C_{r_k}-C_{r_{k+1}}\right)E^k-\sigma\kappa_\Phi\alpha_k\|d^k\|_2^2\\&+
C_{r_{k+1}}\frac{L_F^2}{b_{k+1}}\alpha_k^2\|d^k\|_2^2-\sigma\rho_{\mathrm{gap}}
\alpha_kV(x^k).\end{aligned}\]
Since $C_{r_k}-C_{r_{k+1}}=\omega,\qquad C_{r_{k+1}}\leq q\omega,\qquad b_{k+1}\geq b_0, \qquad 0<\alpha_k\leq1,$
we have
\[C_{r_{k+1}}\frac{L_F^2}{b_{k+1}}\alpha_k^2\leqq\omega\frac{L_F^2}{b_0}\alpha_k
\leq\frac{\sigma\kappa_\Phi}{2}\alpha_k.\]
Therefore,
\[\begin{aligned}\mathbb E\left[\Psi^{k+1}\mid\mathcal F_k\right]\leq{}&\Psi^k-
\frac{\sigma\kappa_\Phi}{2}\alpha_k\|d^k\|_2^2-\omega E^k\\&-\sigma\rho_{\mathrm{gap}}\alpha_kV(x^k).\end{aligned}\]
\medskip
\noindent
\textbf{Case 2: \(k+1\) is a refresh iteration.}
In this case, $E^{k+1}=0,$ and $C_{r_k}=\omega.$

The Armijo decrease and Lemma~\ref{lem:4.7} give
\[\Phi(x^{k+1};\bar\rho)\leq\Phi(x^k;\bar\rho)-\sigma\kappa_\Phi\alpha_k\|d^k\|_2^2
-\sigma\rho_{\mathrm{gap}}\alpha_kV(x^k).\]
Therefore,
\[\begin{aligned}\Psi^{k+1}&=\Phi(x^{k+1};\bar\rho)\\&\leq\Psi^k-\omega E^k-
\sigma\kappa_\Phi\alpha_k\|d^k\|_2^2-\sigma\rho_{\mathrm{gap}}\alpha_kV(x^k).
\end{aligned}\]
Since $\sigma\kappa_\Phi\geq\frac{\sigma\kappa_\Phi}{2},$ the same bound obtained in the non-refresh case also holds here.This proves \eqref{eq:coupled_lyapunov_descent}.
\end{proof}
\begin{corollary}[Almost-sure summability and estimator consistency]
\label{cor:4.2}
Under the conditions of Theorem~\ref{thm:4.2},
\[\sum_{k=K_L}^{\infty}\alpha_k\|d^k\|_2^2<\infty,\qquad\sum_{k=K_L}^{\infty}
E^k<\infty,\qquad\sum_{k=K_L}^{\infty}\alpha_kV(x^k)<\infty\qquad\text{a.s.}\]
Consequently,
\[\sum_{k=K_L}^{\infty}\|x^{k+1}-x^k\|_2^2<\infty\qquad\text{a.s.},\ \ and \quad
\|g^k-\nabla F(x^k)\|_2\longrightarrow0\qquad\text{a.s.}\]
\end{corollary}
\begin{proof}
By Assumption~\ref{ass:4.4}, \(x^k\) remains in the compact set \(\mathcal X\). Hence \(\Phi(x;\bar\rho)\) is bounded below on \(\mathcal X\), and since \(E^k\geq0\), so is \(\Psi^k\). Applying the Robbins--Siegmund theorem to
\eqref{eq:coupled_lyapunov_descent} yields the almost-sure summability of its three nonnegative descent terms.

Furthermore,
\[\|x^{k+1}-x^k\|_2^2=\alpha_k^2\|d^k\|_2^2\leq\alpha_k\|d^k\|_2^2,\]
so the iterate increments are square summable. Finally, \(E^k=\|g^k-\nabla F(x^k)\|_2^2\), and summability of the nonnegative sequence \(\{E^k\}\) implies \(E^k\to0\) almost surely, proving the estimator-consistency conclusion.
\end{proof}
\subsection{Global KKT Convergence}\label{subsec:global_kkt}
We now use the Lyapunov estimates to establish asymptotic primal feasibility, vanishing search directions, dual feasibility, complementarity, and stationarity. The multiplier handoff is used explicitly so that the limiting KKT conditions correspond to the outer primal--dual sequence generated by AVR-SS(QC)QP.
\begin{lemma}[Uniform lower bound on accepted step sizes] \label{lem:4.9}
Suppose the conditions of Theorem~\ref{thm:4.2} hold. Then there exist a constant \(\underline{\alpha}>0\) and a finite index \(K_\alpha\) such that
\begin{equation}\label{eq:uniform_step_lower_bound}
\alpha_k\geq\underline{\alpha},\qquad k\geq K_\alpha.
\end{equation}\end{lemma}
\begin{proof}
By Theorem~\ref{thm:4.1}, the penalty parameter is constant on the relevant tail:$\rho_k=\bar\rho.$ Hence Lemma~\ref{lem:4.3} admits the uniform quadratic-remainder constant $C_\Phi=\frac{L_F}{2}+\frac{\bar\rho}{2}\left(mM_h+pM_c \right).$

Moreover, Lemma~\ref{lem:4.7} gives $-D_k^\Phi\geq \kappa_\Phi\|d^k\|_2^2+\rho_{\mathrm{gap}}V(x^k)\geq\kappa_\Phi\|d^k\|_2^2.$
By the estimator-accuracy condition used in Lemma~\ref{lem:4.5}, there exists
\(0\leq\vartheta<1-\sigma\) such that $\|e^k\|_2\|d^k\|_2\leq\vartheta(-D_k^\Phi)$
on the relevant tail. Therefore the stochastic merit-remainder boundimplies
\[\Phi(x^k+\alpha d^k;\bar\rho)\leq\Phi(x^k;\bar\rho)+\alpha(1-\vartheta)D_k^\Phi+
C_\Phi\alpha^2\|d^k\|_2^2.\]
Consequently, the Armijo condition is guaranteed whenever
\[C_\Phi\alpha\|d^k\|_2^2\leq(1-\sigma-\vartheta)(-D_k^\Phi).\]
Using $-D_k^\Phi\geq\kappa_\Phi\|d^k\|_2^2,$
it is sufficient that\[\alpha\leq\frac{(1-\sigma-\vartheta)\kappa_\Phi}{C_\Phi}.\]

By Lemma~\ref{lem:4.6}, the effective directions are bounded on the regular tail. Since \(\mathcal X\subset\mathcal N\), with \(\mathcal X\) compact and \(\mathcal N\) open, there exists a uniform \(\bar\alpha>0\) such that all trial segments
$x^k+\alpha d^k,\ 0\leq\alpha\leq\bar\alpha,$ remain in the neighborhood where the remainder estimate is valid.

Define \[\alpha_\star=\min\left\{1,\bar\alpha,\frac{(1-\sigma-\vartheta)\kappa_\Phi}{C_\Phi}\right\}>0.\]
Every trial step satisfying \(\alpha\leq\alpha_\star\) therefore satisfies the Armijo condition.

Since backtracking tests $1,\beta,\beta^2,\ldots,\ \ 0<\beta<1,$ the first accepted step satisfies $\alpha_k\geq\beta\alpha_\star.$
Thus \eqref{eq:uniform_step_lower_bound} holds with
\[\underline\alpha=\beta\alpha_\star>0.\]
\end{proof}

\begin{lemma}[Vanishing direction and primal infeasibility]\label{lem:4.10}
Under the conditions of Theorem~\ref{thm:4.2},
\begin{equation}\label{eq:direction_and_feasibility_zero}
\|d^k\|_2\longrightarrow0,\qquad V(x^k)\longrightarrow0\qquad\text{a.s.}
\end{equation}
Consequently,
\[h(x^k)\longrightarrow0,\qquad c(x^k)^+\longrightarrow0\qquad\text{a.s.}\]
\end{lemma}
\begin{proof}
Corollary~\ref{cor:4.2} gives $\sum_{k=K_L}^{\infty}\alpha_k\|d^k\|_2^2<\infty\ a.s.$ and $\sum_{k=K_L}^{\infty}\alpha_kV(x^k)<\infty\ a.s.$

By Lemma~\ref{lem:4.9},$\alpha_k\geq\underline\alpha>0$ for all sufficiently large \(k\). Hence \[\underline\alpha\sum_{k=K_\alpha}^{\infty}\|d^k\|_2^2\leq \sum_{k=K_\alpha}^{\infty}\alpha_k\|d^k\|_2^2<\infty.\]
Therefore $\|d^k\|_2\longrightarrow0\qquad\text{a.s.}$

Similarly,\[\underline\alpha\sum_{k=K_\alpha}^{\infty}V(x^k)\leq \sum_{k=K_\alpha}^{\infty}\alpha_kV(x^k)<\infty.\]
Since \(V(x^k)\geq0\),\[V(x^k)\longrightarrow0\qquad\text{a.s.}\]
By the definition $V(x)=\|h(x)\|_1+\|c(x)^+\|_1,$ both nonnegative terms must converge to zero, which proves the result.
\end{proof}
\begin{lemma}[Asymptotic dual feasibility and complementarity]\label{lem:4.11}
Under Assumptions~\ref{ass:4.1}--\ref{ass:4.5} and the multiplier handoff rule,
$\mu^{k+1}=\widehat{\mu}^k,$ the outer inequality multipliers satisfy $\mu^k\geq0$ for every sufficiently large \(k\), and
\begin{equation}\label{eq:complementarity_residual_zero}
\|\mu^k\odot c(x^k)\|_2\longrightarrow0\qquad\text{a.s.}\end{equation}\end{lemma}
\begin{proof}
The inequality multipliers of each local SQP or SQCQP subproblem satisfy $\widehat{\mu}^k\geq0.$ Since the accepted local multiplier pair is handed directly to the outer sequence, $\mu^{k+1}=\widehat{\mu}^k,$ dual feasibility is inherited by the outer multipliers.

It remains to establish complementary slackness. Let \(\widehat c_j^k(d^k)\) denote the local model of the \(j\)-th inequality constraint associated with the effective direction:
\[ \widehat{c}_j^k(d^k) = c_j(x^k) + \nabla c_j(x^k)^\top d^k \quad \text{in SQP mode, and} \]
\[ \widehat{c}_j^k(d^k) = c_j(x^k) + \nabla c_j(x^k)^\top d^k + \frac{1}{2} (d^k)^\top \nabla^2 c_j(x^k) d^k \quad \text{in SQCQP mode.} \]

The local KKT complementarity condition gives $\widehat\mu_j^k\widehat c_j^k(d^k)=0.$ Consequently, in either mode,
\[\left|\widehat\mu_j^k c_j(x^k)\right|\leq|\widehat\mu_j^k|\left(\|\nabla c_j(x^k)\|_2\|d^k\|_2+\frac12M_c\|d^k\|_2^2\right),\]
where the quadratic term is absent in SQP mode.

By Lemma~\ref{lem:4.6}, the local multipliers are uniformly bounded, and by Lemma~\ref{lem:4.10},\[d^k\longrightarrow0.\]
Hence \[\widehat\mu^k\odot c(x^k)\longrightarrow0.\]

Using multiplier handoff,$\mu^{k+1}=\widehat\mu^k.$ and $x^{k+1}-x^k=\alpha_kd^k
\longrightarrow0$ moreover \(c\) is continuously differentiable on the compact relevant region, $c(x^{k+1})-c(x^k)\longrightarrow0.$
Therefore
\[\begin{aligned}\mu^{k+1}\odot c(x^{k+1})&=\widehat\mu^k\odot c(x^k)\ +\
\widehat\mu^k\odot\left(c(x^{k+1})-c(x^k)\right)\longrightarrow0,\end{aligned}\]
because \(\widehat\mu^k\) is bounded. Shifting the index proves \eqref{eq:complementarity_residual_zero}.\end{proof}

\begin{theorem}[Global KKT convergence]\label{thm:4.3}
Suppose Assumptions~\ref{ass:4.1}--\ref{ass:4.5} hold together with the hypotheses of Theorem~\ref{thm:4.2}. Then the KKT residual of the outer AVR-SS(QC)QP sequence satisfies
\begin{equation}\label{eq:global_kkt_residual}
\left\|R_{\mathrm{KKT}}(x^k,\lambda^k,\mu^k)\right\|_2\longrightarrow0 \qquad\text{a.s.}\end{equation}
Consequently, every accumulation point $(x^\star,\lambda^\star,\mu^\star)$ of the primal--dual sequence satisfies
\[\nabla_x\mathcal L(x^\star,\lambda^\star,\mu^\star)=0,\]
\[h(x^\star)=0,\qquad c(x^\star)\leq0,\]
\[\mu^\star\geq0,\qquad\mu^\star\odot c(x^\star)=0\]
almost surely.
\end{theorem}
\begin{proof}
Lemma~\ref{lem:4.10} gives \[h(x^k)\longrightarrow0,\qquad c(x^k)^+\longrightarrow0 \qquad\text{a.s.}\]
Thus the primal-feasibility blocks of the KKT residual vanish.
Lemma~\ref{lem:4.11} gives \[\mu^k\geq0\]
and \[\mu^k\odot c(x^k)\longrightarrow0\qquad\text{a.s.}\]
It remains only to establish stationarity.

Let $(\widehat\lambda^k,\widehat\mu^k)$ be the local multiplier pair corresponding to the effective direction \(d^k\). The stationarity condition of the effective local subproblem is written as
\[g^k+J_h(x^k)^\top\widehat\lambda^k+J_c(x^k)^\top\widehat\mu^k+Q^k d^k+
r_{\mathrm{curv}}^k=0,\] where
\[Q^k=\begin{cases}B^k,&\text{for an effective SQP direction},\\G^k,&\text{for an effective SQCQP direction},\end{cases}\]
and\[r_{\mathrm{curv}}^k=\begin{cases}0,&\text{for SQP},\\H^k(\widehat\nu^k)d^k,
&\text{for SQCQP}.\end{cases}\]

By Assumption~\ref{ass:4.3}, both \(B^k\) and \(G^k\) are uniformly bounded. Hence Lemma~\ref{lem:4.10} gives $Q^k d^k\longrightarrow0.$ By Lemma~\ref{lem:4.6} and Assumption~\ref{ass:4.1}, $\|H^k(\widehat\nu^k)\|_2\leq C_H$ on the relevant tail, and consequently \[r_{\mathrm{curv}}^k\longrightarrow0.\]
It follows that\[g^k+J_h(x^k)^\top\widehat\lambda^k+J_c(x^k)^\top\widehat\mu^k
\longrightarrow0.\]

Using Corollary~\ref{cor:4.2} we get,\[\nabla F(x^k)+J_h(x^k)^\top\widehat\lambda^k+J_c(x^k)^\top\widehat\mu^k\longrightarrow0.\]

After the accepted step, the multiplier handoff gives
$\lambda^{k+1}=\widehat\lambda^k,\ \mu^{k+1}= \widehat\mu^k.$
Furthermore,$x^{k+1}-x^k=\alpha_kd^k\longrightarrow0.$
The continuity of \(\nabla F\), \(J_h\), and \(J_c\), together with the boundedness of the multipliers, therefore gives
\[\begin{aligned}&\nabla F(x^{k+1})+J_h(x^{k+1})^\top\lambda^{k+1}+ J_c(x^{k+1})^\top\mu^{k+1} \longrightarrow0.\end{aligned}\]
Hence \[\left\|\nabla_x\mathcal L (x^{k+1},\lambda^{k+1},\mu^{k+1})\right\|_2\longrightarrow0,\]
and shifting the index yields
\[\left\|\nabla_x\mathcal L(x^k,\lambda^k,\mu^k)\right\|_2\longrightarrow0.\]

Thus every block of \eqref{eq:kkt_residual} converges to zero almost surely. This proves\eqref{eq:global_kkt_residual}.

Finally, Assumption~\ref{ass:4.4} and Lemma~\ref{lem:4.6} imply boundedness of the primal--dual sequence. Hence accumulation points exist. Let
$(x^{k_j},\lambda^{k_j},\mu^{k_j})\longrightarrow(x^\star,\lambda^\star,\mu^\star).$
Passing to the limit in the four vanishing KKT-residual blocks gives
\[\nabla_x\mathcal L(x^\star,\lambda^\star,\mu^\star)=0,\]
\[h(x^\star)=0,\qquad c(x^\star)\leq0,\]
and
\[\mu^\star\geq0,\qquad\mu^\star\odot c(x^\star)=0.\]
Therefore every accumulation point is a KKT point almost surely.
\end{proof}
\subsection{Eventual SQP Selection}
\label{subsec:eventual_sqp}
The global KKT convergence result now permits an analysis of the asymptotic behavior of the RCMCI switching mechanism. The essential observation is that the stochastic SQP probe vanishes in the regular KKT regime. Consequently, the second-order constraint corrections measured by the RCMCI vanish quadratically, forcing the switching mechanism eventually to select the SQP model.
\begin{lemma}[Vanishing SQP probe]\label{lem:4.12}
Suppose Assumptions~\ref{ass:4.1}--\ref{ass:4.5} and the hypotheses of Theorem~\ref{thm:4.3} hold. Then
\begin{equation}\label{eq:vanishing_sqp_probe}
\|d_{\mathrm{SQP}}^k\|_2\longrightarrow0\qquad\text{a.s.}\end{equation}\end{lemma}
\begin{proof}
Consider an almost-sure sample path on which the conclusions of Theorem~\ref{thm:4.3} and Corollary~\ref{cor:4.2} hold.In particular,
\[\left\|R_{\mathrm{KKT}}(x^k,\lambda^k,\mu^k)\right\|_2\longrightarrow0\]
and \[g^k-\nabla F(x^k)\longrightarrow0.\]
Suppose, to the contrary, that \(\|d_{\mathrm{SQP}}^k\|_2\not\to0\). Then there exist \(\varepsilon_d>0\) and an infinite subsequence \(\{k_j\}\) such that
\[\|d_{\mathrm{SQP}}^{k_j}\|_2\geq\varepsilon_d.\]
By Assumption~\ref{ass:4.4}, the primal sequence lies in the compact set \(\mathcal X\). By Lemma~\ref{lem:4.6}, the multiplier sequences are bounded. Hence, after passing to a further subsequence if necessary,
\[x^{k_j}\longrightarrow x^\star,\qquad \lambda^{k_j}\longrightarrow\lambda^\star,
\qquad \mu^{k_j}\longrightarrow\mu^\star.\]
Theorem~\ref{thm:4.3} then implies that \((x^\star,\lambda^\star,\mu^\star)\) is a KKT point.

Moreover,\[g^{k_j}=\nabla F(x^{k_j})+e^{k_j}\longrightarrow\nabla F(x^\star).\]
Since there are only finitely many possible active sets, a further subsequence may be chosen on which the locally identified active set is fixed. Assumption~\ref{ass:4.5} gives strong regularity of the corresponding local SQP KKT system with a uniformly bounded inverse-Lipschitz modulus.

At the limiting KKT point, the SQP subproblem admits the zero direction
as the local solution. Indeed, $h(x^\star)=0,\ c_{\mathcal A^\star}(x^\star)=0,$
and the KKT stationarity condition gives
\[\nabla F(x^\star)+J_h(x^\star)^\top\lambda^\star+J_{c_{\mathcal A^\star}}(x^\star)^\top\mu_{\mathcal A^\star}^\star=0.\]
Thus $d_{\mathrm{SQP}}^\star=0$ solves the limiting SQP KKT system.

By strong regularity, the locally relevant SQP primal--dual solution
depends Lipschitz continuously on the local problem data. Since those
data converge to their limiting KKT values along the selected
subsequence, $d_{\mathrm{SQP}}^{k_j}\longrightarrow d_{\mathrm{SQP}}^\star=0.$

This contradicts $\|d_{\mathrm{SQP}}^{k_j}\|_2\geq\varepsilon_d.$ Therefore
\[\|d_{\mathrm{SQP}}^k\|_2\longrightarrow0\qquad\text{a.s.}\]
\end{proof}
\begin{lemma}[Vanishing RCMCI]\label{lem:4.13}
Under the conditions of Lemma~\ref{lem:4.12},
\begin{equation}\label{eq:vanishing_raw_rcmci}
\eta_k\longrightarrow0 \qquad\text{a.s.} \end{equation}
and\begin{equation}\label{eq:vanishing_smoothed_rcmci}
\bar{\eta}_k\longrightarrow0 \qquad\text{a.s.} \end{equation}\end{lemma}
\begin{proof}
By Assumption~\ref{ass:4.1}, the constraint Hessians are uniformly bounded. Hence, for every equality and inequality constraint, \[|q_{h_i,k}|\leq\frac{M_h}{2}\|d_{\mathrm{SQP}}^k\|_2^2,\qquad |q_{c_j,k}| \leq \frac{M_c}{2}\|d_{\mathrm{SQP}}^k\|_2^2.\]
Since every RCMCI normalization scale is bounded below by \(\epsilon_s>0\), it follows that
\[0\leq\eta_k\leq\frac{M_{\mathrm{curv}}}{2\epsilon_s}\|d_{\mathrm{SQP}}^k\|_2^2,
\qquad M_{\mathrm{curv}}:=\max\{M_h,M_c\}.\]
Lemma~\ref{lem:4.12} therefore implies $\eta_k\longrightarrow0\qquad\text{a.s.}$

For the smoothed indicator,\[\bar\eta_k=\beta_\eta\bar\eta_{k-1} +(1-\beta_\eta)\eta_k,\qquad 0\leq\beta_\eta<1.\]
Iterating from any \(K\) gives \[\bar\eta_k=\beta_\eta^{\,k-K+1}\bar\eta_{K-1}+(1-\beta_\eta)\sum_{\ell=K}^{k} \beta_\eta^{\,k-\ell}\eta_\ell.\]
Since \(\eta_\ell\to0\) almost surely and the geometric weights have
total mass at most one, the second term tends to zero, while the first
vanishes because \(\beta_\eta<1\). Hence $\bar\eta_k\longrightarrow0\qquad\text{a.s.}$
\end{proof}
\begin{theorem}[Eventual SQP selection]\label{thm:4.4}
Under the conditions of Theorem~\ref{thm:4.3} and Lemma~\ref{lem:4.12}, there exists a finite random index \(K_{\mathrm{switch}}\) such that
\begin{equation}\label{eq:eventual_sqp_selection}
\mathcal M_k=\mathrm{SQP},\qquad k\geq K_{\mathrm{switch}},\qquad \text{a.s.}
\end{equation}
Consequently, the effective direction passed to globalization is also the SQP direction for every sufficiently large iteration.
\end{theorem}
\begin{proof}
By Lemma~\ref{lem:4.13}, $\bar{\eta}_k\longrightarrow0\qquad\text{a.s.} \ $
and the lower hysteresis threshold satisfies $\eta_{\mathrm{low}}>0.$
Therefore, on every sample path for which \(\bar{\eta}_k\to0\), there exists a finite index \(K_{\mathrm{switch}}\) such that $\bar{\eta}_k\leq\eta_{\mathrm{low}},\ k\geq K_{\mathrm{switch}}.$
By the RCMCI switching rule, this implies \[\mathcal M_k=\mathrm{SQP}\]
for every \(k\geq K_{\mathrm{switch}}\).

Since the globalization safeguard can replace an SQCQP candidate by
the SQP probe but does not replace an SQP preference by an SQCQP
direction, eventual SQP preference also implies eventual SQP effective
mode.
\end{proof}

\begin{corollary}[Vanishing asymptotic SQCQP frequency]\label{cor:4.3}
For an infinite generated sequence, define
$N_{\mathrm{SQCQP}}(K)=\#\left\{0\leq k<K:\mathcal M_k=\mathrm{SQCQP}\right\}.$
Under the conditions of Theorem~\ref{thm:4.4},
\begin{equation}\label{eq:sqcqp_frequency_zero}
\lim_{K\to\infty}\frac{N_{\mathrm{SQCQP}}(K)}{K}=0 \qquad\text{a.s.}\end{equation}
\end{corollary}

\begin{proof}
By Theorem~\ref{thm:4.4}, there exists a finite random index 
\(K_{\mathrm{switch}}\) after which the RCMCI always selects SQP. Hence
$N_{\mathrm{SQCQP}}(K)\leq K_{\mathrm{switch}}$ for all sufficiently large \(K\). Therefore,\[0\leq\frac{N_{\mathrm{SQCQP}}(K)}{K}\leq\frac{K_{\mathrm{switch}}}{K}
\longrightarrow0\qquad\text{a.s.}\]
\end{proof}
\section{Numerical Experiments}\label{sec:numerical_experiments}
This section evaluates the proposed Adaptive Variance-Reduced Stochastic Sequential (Quadratic Constrained) Quadratic Programming (AVR-SS(QC)QP) framework on a constrained Sparse Principal Component Analysis (Sparse PCA) problem. The experiments assess objective quality, first-order KKT stationarity, constraint satisfaction, stochastic-gradient efficiency, and wall-clock performance. The
numerical behavior is also interpreted in relation to the SPIDER variance-reduction and adaptive SQP--SQCQP model-selection mechanisms.
\subsection{Experimental Setup and Protocol}
\label{subsec:experimental_setup}
The experiments use the \emph{Student Performance Factors} dataset, which contains \(N=6607\) observations. After preprocessing, seven numerical variables are retained and standardized to zero empirical mean and unit standard deviation. Let
\(X\in\mathbb{R}^{N\times7}\) denote the standardized data matrix and
\(\Sigma\) its empirical covariance matrix. The constrained Sparse PCA problem is
$\min_{x\in\mathbb{R}^{7}}\quad F(x)=-x^\top\Sigma x\qquad \text{s.t.}
\qquad \|x\|_2^2-1=0, \qquad \sum_{i=1}^{7}\sqrt{x_i^2+\delta^2}-\tau\leq0,$
where the sparsity budget and smoothing parameter are fixed at \(\tau=2\) and \(\delta=10^{-6}\), respectively. The equality constraint normalizes the loading vector, whereas the smooth \(\ell_1\)-type inequality promotes sparsity.

All methods are initialized from the common deterministic point \(x^0=e_1\), with zero initial multipliers whenever required. For AVR-SS(QC)QP, the initial curvature matrices are \(B^0=G^0=I\). The proposed method is compared with three stochastic constrained optimization baselines: AALM-LO, representing an adaptive
stochastic augmented-Lagrangian approach; StoSQP-CRZ, representing a stochastic SQP framework; and Constrained-Adam, representing an adaptive first-order constrained method. The baseline methods use a fixed mini-batch size of \(256\), whereas AVR-SS(QC)QP employs adaptive mini-batching.

The principal AVR-SS(QC)QP parameters are
$\rho_0=1,\ b_0=100,\ b_{\max}=N,\ \gamma_b=0.1,\ q=5,\ \eta_{\mathrm{low}}=0.01,\
\eta_{\mathrm{high}}=0.10,\ \beta_\eta=0.5,\ \gamma_d=1,\ \epsilon_s=10^{-8},\
\delta_B=\delta_G=0.2,\ \sigma=10^{-4},\ \beta=0.5,$ and $ \gamma_\rho=2.$
The numerical SQCQP tolerance is \(10^{-6}\), while the KKT termination tolerance is \(10^{-5}\). For the principal baseline configurations, AALM-LO uses initial penalty \(\beta_0=1\), penalty factors \(0.5\) and \(10\), and initial inner learning rate \(10^{-2}\). StoSQP-CRZ uses \(\tau_0=0.1\), \(\xi_0=1\), \(\sigma=0.1\), \(\eta=0.5\), and \(\theta=10^4\). Constrained-Adam uses
\(\beta_1=0.9\), \(\beta_2=0.999\), \(\epsilon_{\mathrm{Adam}}=10^{-8}\), and diagonal KKT scale \(10^3\).

Each stochastic method is evaluated using the five random seeds \(\{1,2,3,4,5\}\). Unless convergence occurs earlier, each run is terminated after \(300\) outer iterations. Convergence is declared when $\left\|R_{\mathrm{KKT}}(x^k,\lambda^k,\mu^k)\right\|_2\leq10^{-5}.$
The objective value \(F(x^k)=-(x^k)^\top\Sigma x^k\) measures the variance captured by the loading vector, with more negative values corresponding to larger
captured variance. Since objective value alone does not establish solution quality in a constrained problem, it is interpreted jointly with the aggregate constraint violation $V(x^k)=\|h(x^k)\|_1+\|c(x^k)^+\|_1$ and the exact KKT residual. Computational cost is measured using both wall-clock runtime and cumulative algorithmic sample-gradient evaluations. Exact objective and full-gradient calculations performed solely for diagnostic KKT evaluation are excluded from the stochastic-gradient oracle counts.

For AVR-SS(QC)QP, the implementation additionally records the raw and smoothed RCMCI values, selected SQP/SQCQP mode, effective globalization mode, mini-batch size, accepted step length, penalty parameter, and local curvature information. These quantities are used as internal diagnostics for verifying the model-selection and variance-reduction mechanisms, while the principal cross-method
comparison is based on the common metrics reported below.

All computations were implemented in Python~3.13 using NumPy, Pandas, SciPy, and Matplotlib and were performed on an Intel Core i5-8265U processor with \(8\) GB RAM. The solvers were run sequentially to avoid inter-solver CPU contention in the runtime measurements. For a scalar metric \(Y\), the reported mean is computed over the five independent runs, and the \(95\%\) confidence interval is based on the Student-\(t\) distribution with four degrees of freedom,
$\overline Y\pm t_{0.975,4}\frac{S_Y}{\sqrt{5}},$ where \(S_Y\) denotes the sample standard deviation.

\begin{figure*}[p]
\centering

\begin{minipage}[t]{0.48\textwidth}
\centering
\includegraphics[
    width=\linewidth,
    height=0.185\textheight,
    keepaspectratio
]{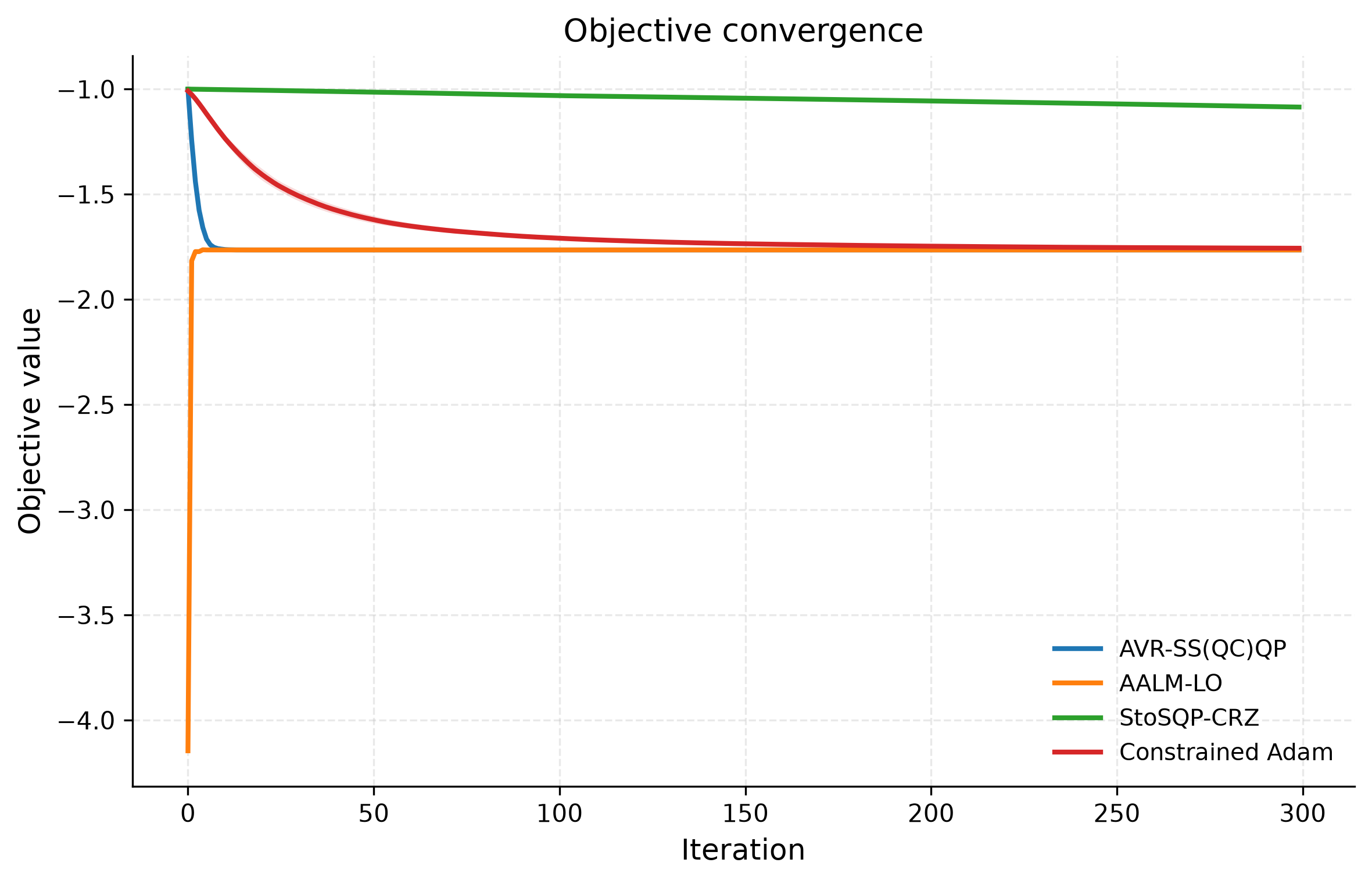}
\vspace{-2mm}

{\small (a) Objective value versus iteration}
\end{minipage}
\hfill
\begin{minipage}[t]{0.48\textwidth}
\centering
\includegraphics[
    width=\linewidth,
    height=0.185\textheight,
    keepaspectratio
]{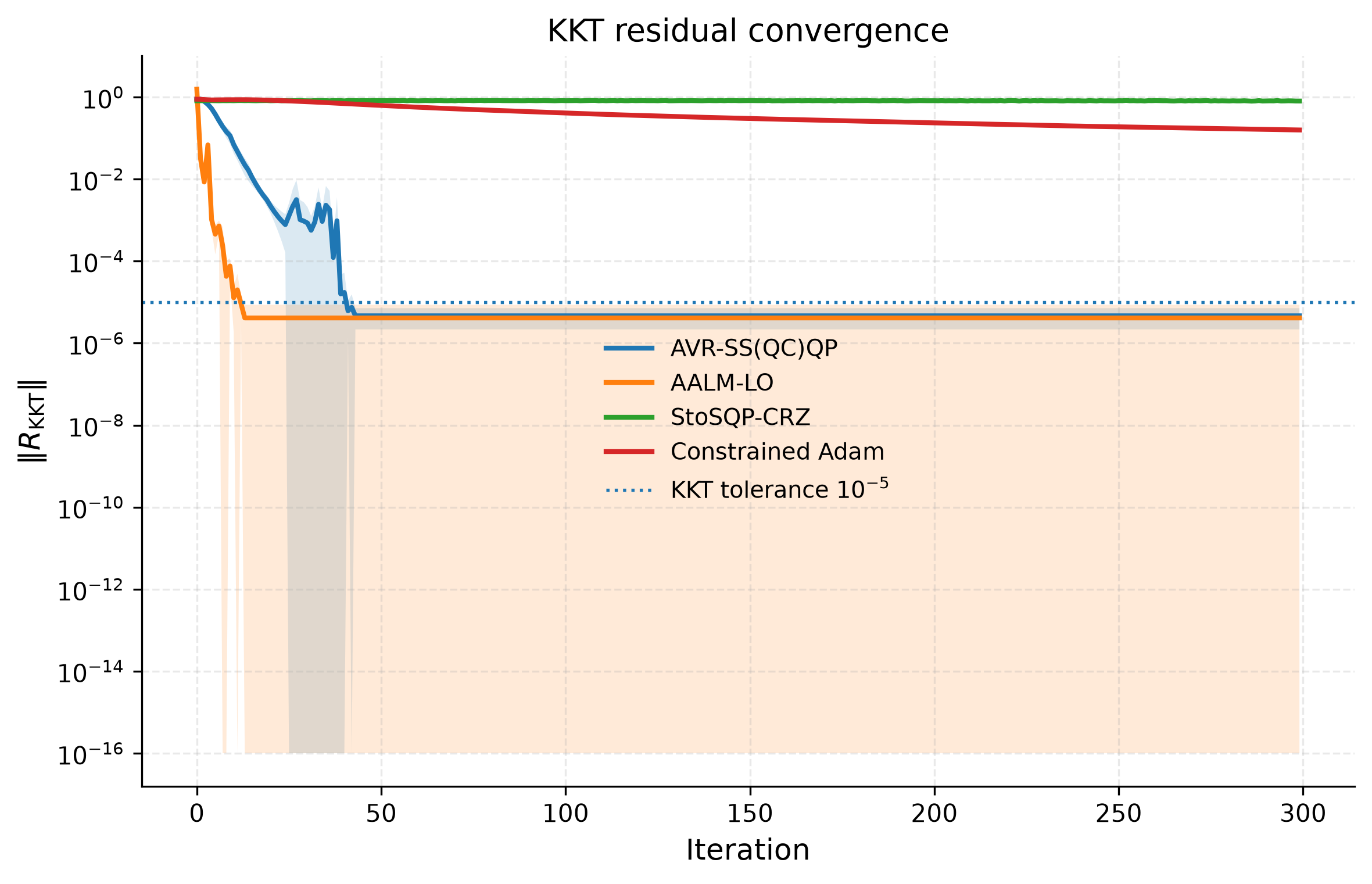}
\vspace{-2mm}

{\small (b) KKT residual versus iteration}
\end{minipage}

\vspace{1.5mm}

\begin{minipage}[t]{0.48\textwidth}
\centering
\includegraphics[
    width=\linewidth,
    height=0.185\textheight,
    keepaspectratio
]{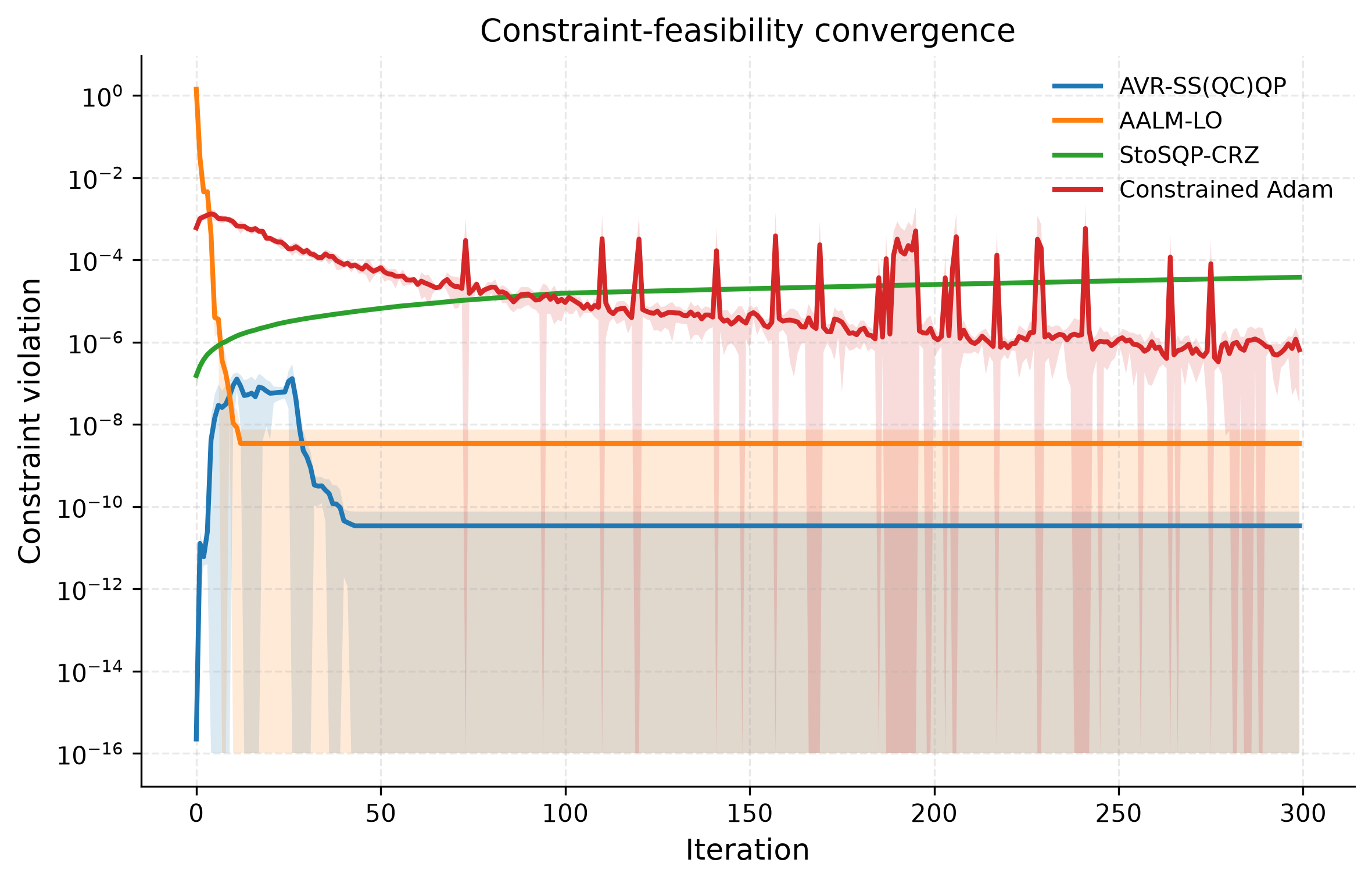}
\vspace{-2mm}

{\small (c) Constraint violation versus iteration}
\end{minipage}
\hfill
\begin{minipage}[t]{0.48\textwidth}
\centering
\includegraphics[
    width=\linewidth,
    height=0.185\textheight,
    keepaspectratio
]{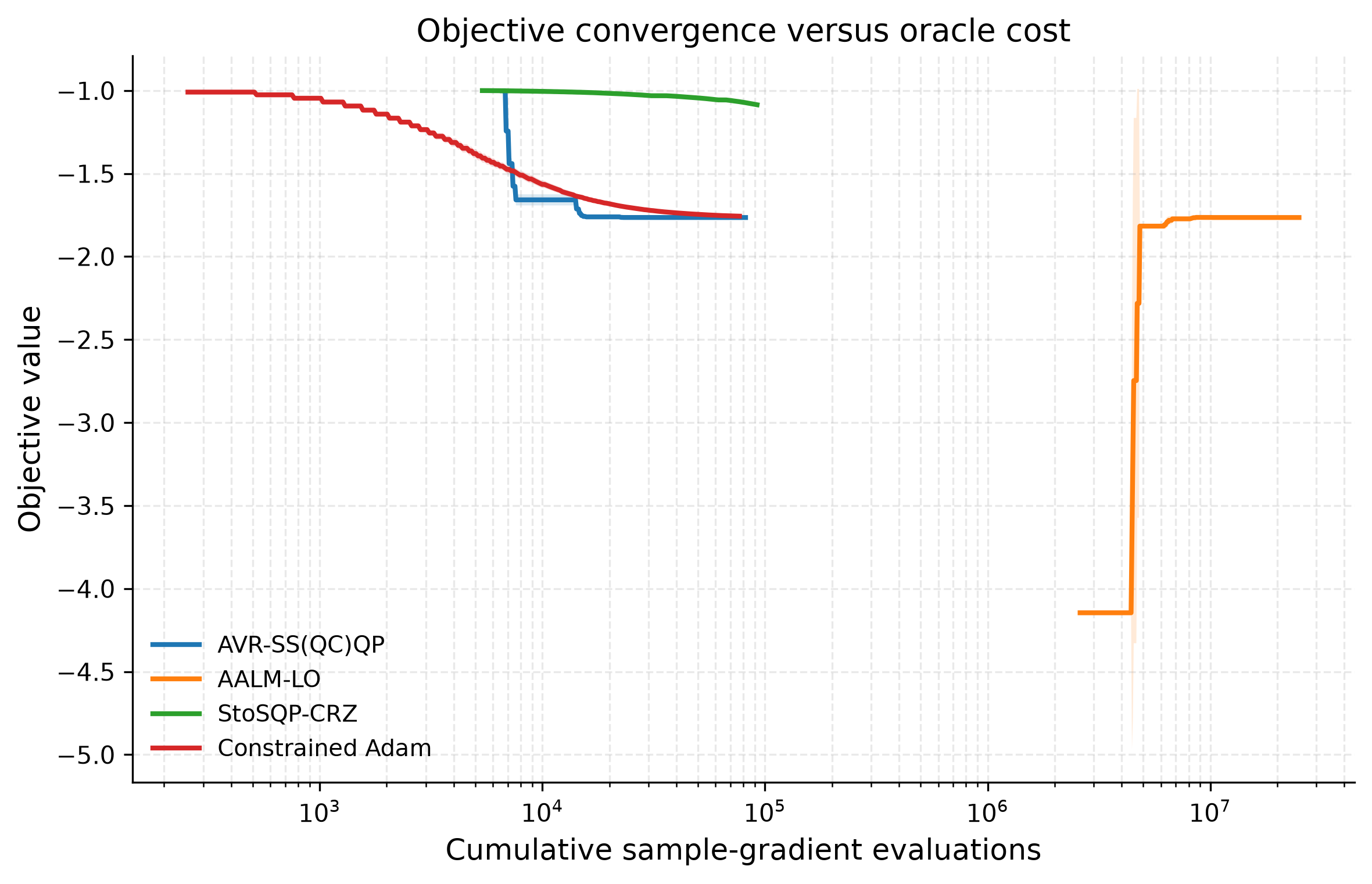}
\vspace{-2mm}

{\small (d) Objective value versus oracle evaluations}
\end{minipage}

\vspace{1.5mm}

\begin{minipage}[t]{0.48\textwidth}
\centering
\includegraphics[
    width=\linewidth,
    height=0.185\textheight,
    keepaspectratio
]{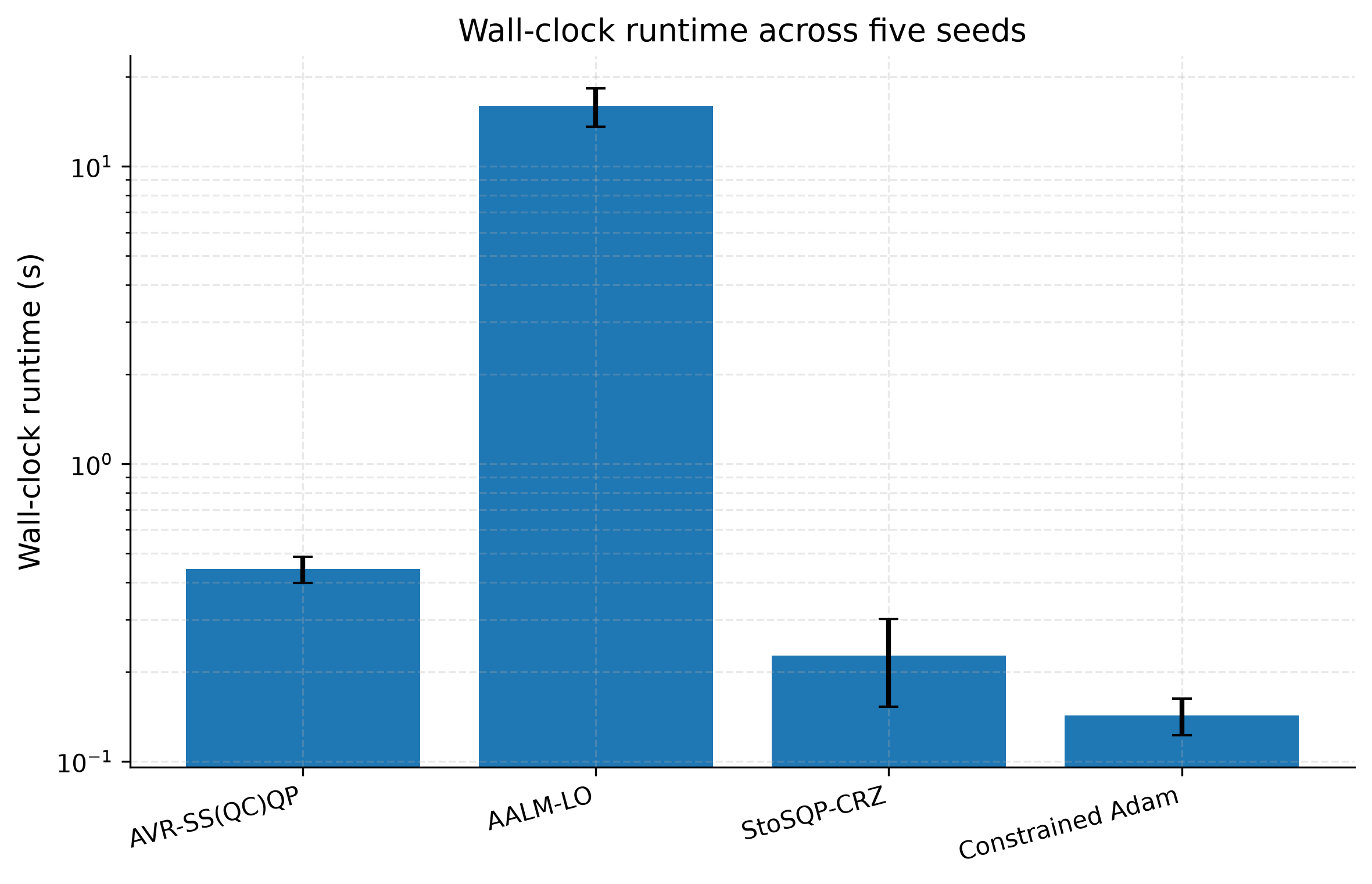}
\vspace{-2mm}

{\small (e) Wall-clock runtime across five seeds}
\end{minipage}
\hfill
\begin{minipage}[t]{0.48\textwidth}
\centering
\includegraphics[
    width=\linewidth,
    height=0.185\textheight,
    keepaspectratio
]{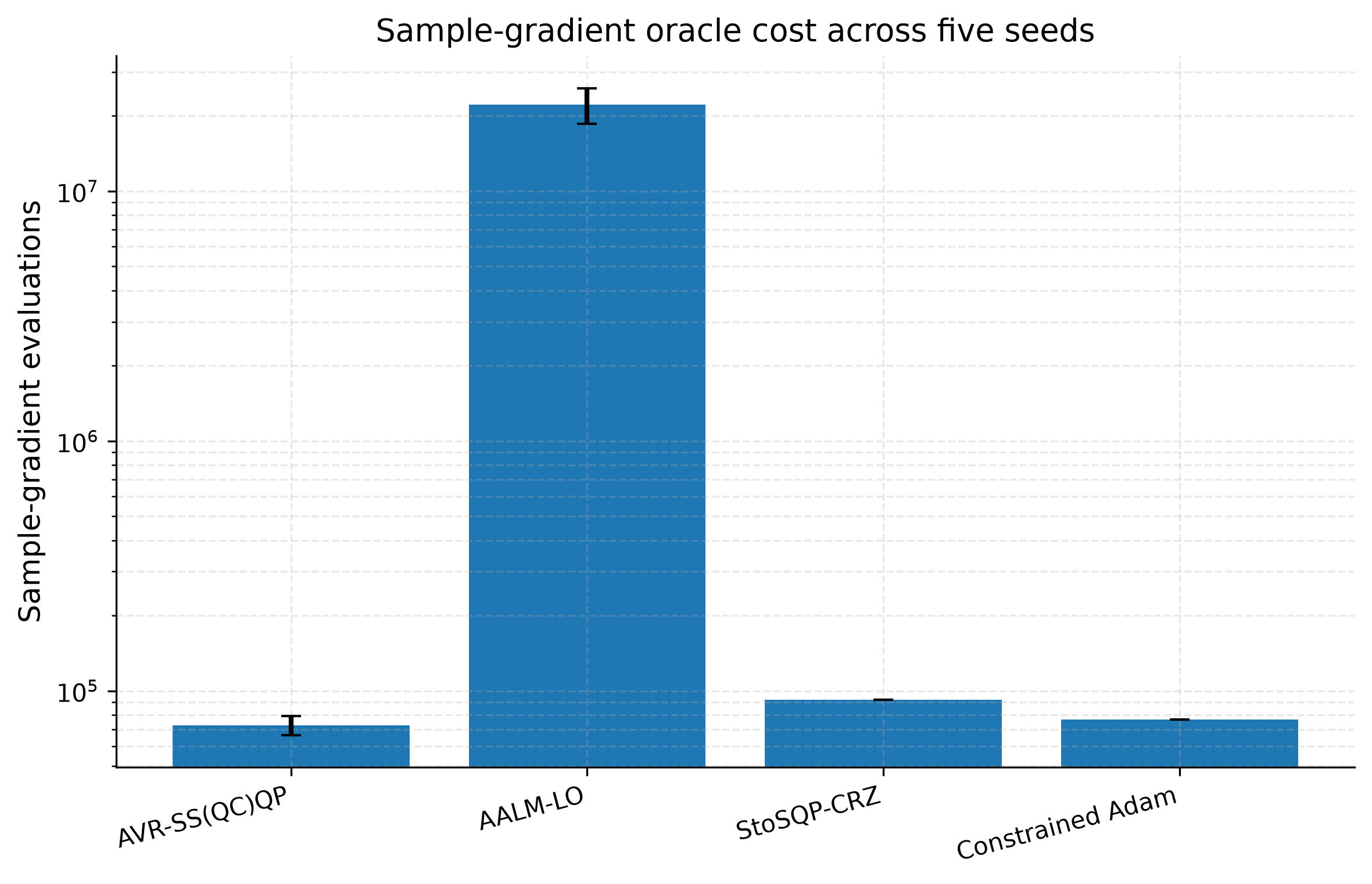}
\vspace{-2mm}

{\small (f) Total sample-gradient oracle cost}
\end{minipage}

\vspace{1.5mm}

\begin{minipage}[t]{0.48\textwidth}
\centering
\includegraphics[
    width=\linewidth,
    height=0.185\textheight,
    keepaspectratio
]{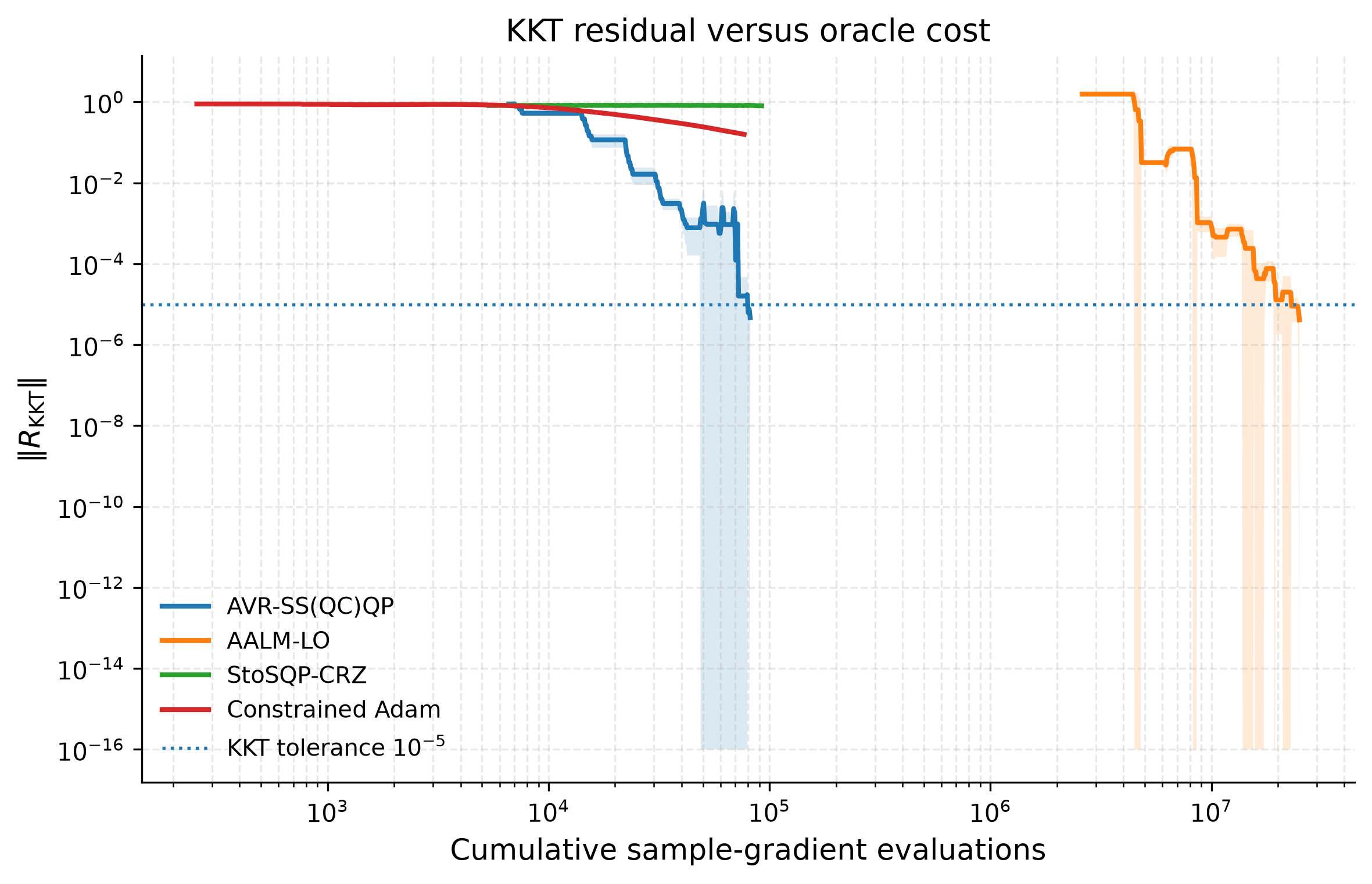}
\vspace{-2mm}

{\small (g) KKT residual versus oracle evaluations}
\end{minipage}
\hfill
\begin{minipage}[t]{0.48\textwidth}
\centering
\includegraphics[
    width=\linewidth,
    height=0.185\textheight,
    keepaspectratio
]{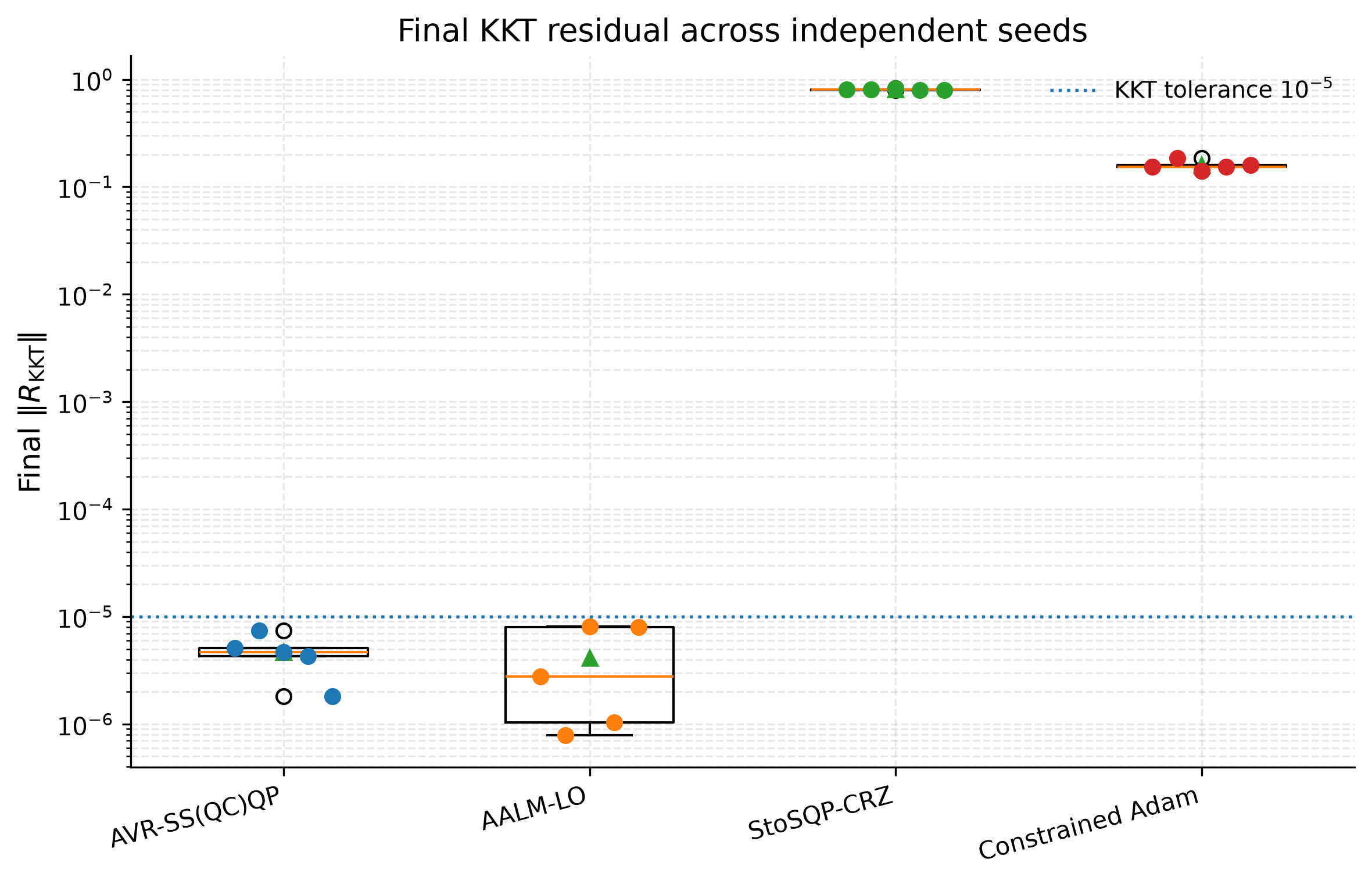}
\vspace{-2mm}

{\small (h) Final KKT residual across independent seeds}
\end{minipage}

\caption{Cross-method numerical comparison on the constrained Sparse
PCA problem. Panels (a)--(c) show objective convergence, KKT
stationarity, and constraint feasibility with respect to outer
iterations; panels (d) and (g) measure convergence with respect to
cumulative sample-gradient evaluations; panels (e)--(f) summarize
runtime and oracle cost; and panel (h) reports the distribution of
final KKT residuals across the five stochastic runs.}
\label{fig:benchmark_summary}
\end{figure*}
\begin{figure*}[p]
\centering

\begin{minipage}[t]{0.48\textwidth}
\centering
\includegraphics[
    width=\linewidth,
    height=0.185\textheight,
    keepaspectratio
]{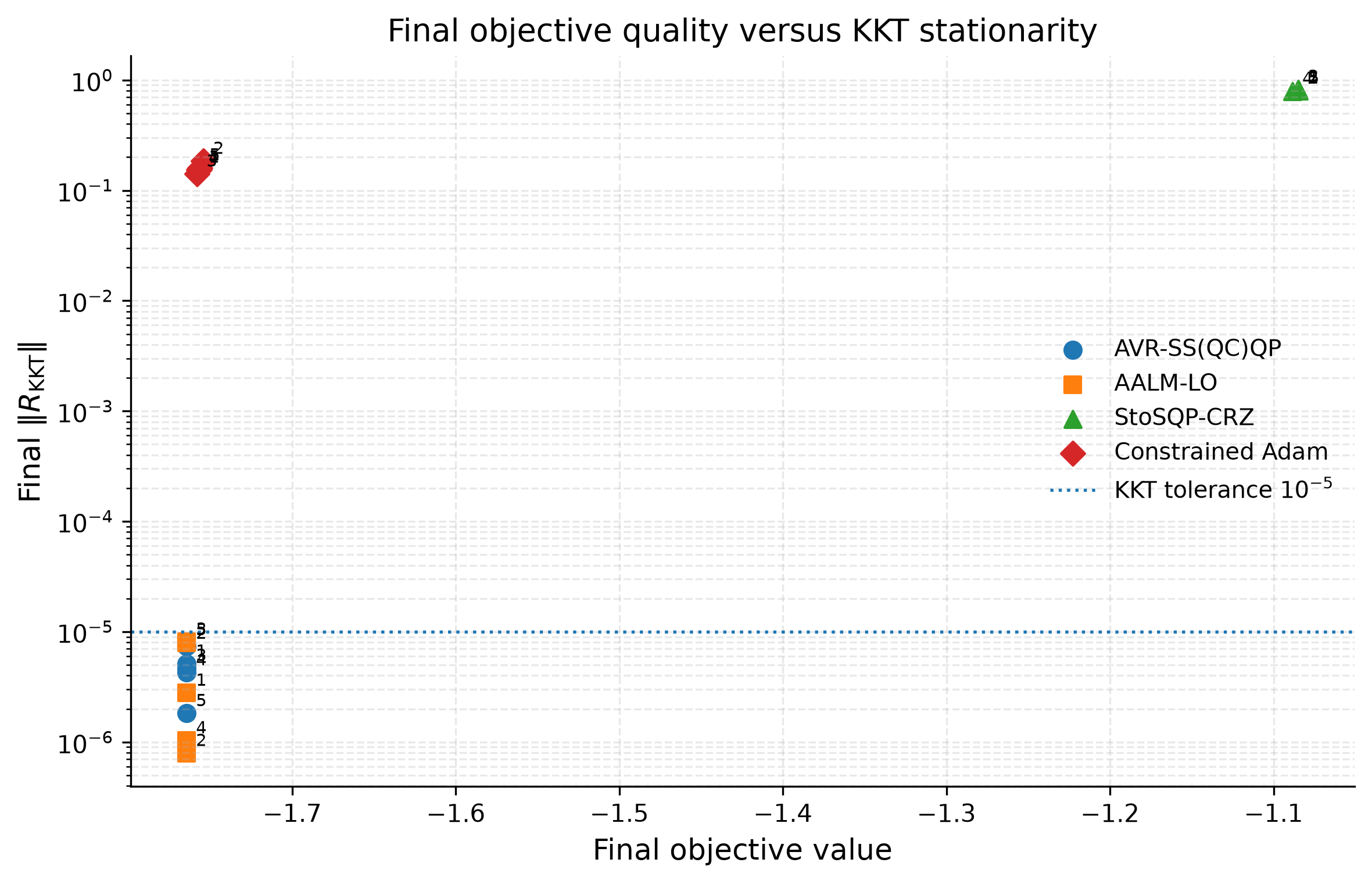}
\vspace{-2mm}

{\small (a) Final objective quality versus KKT stationarity}
\end{minipage}
\hfill
\begin{minipage}[t]{0.48\textwidth}
\centering
\includegraphics[
    width=\linewidth,
    height=0.185\textheight,
    keepaspectratio
]{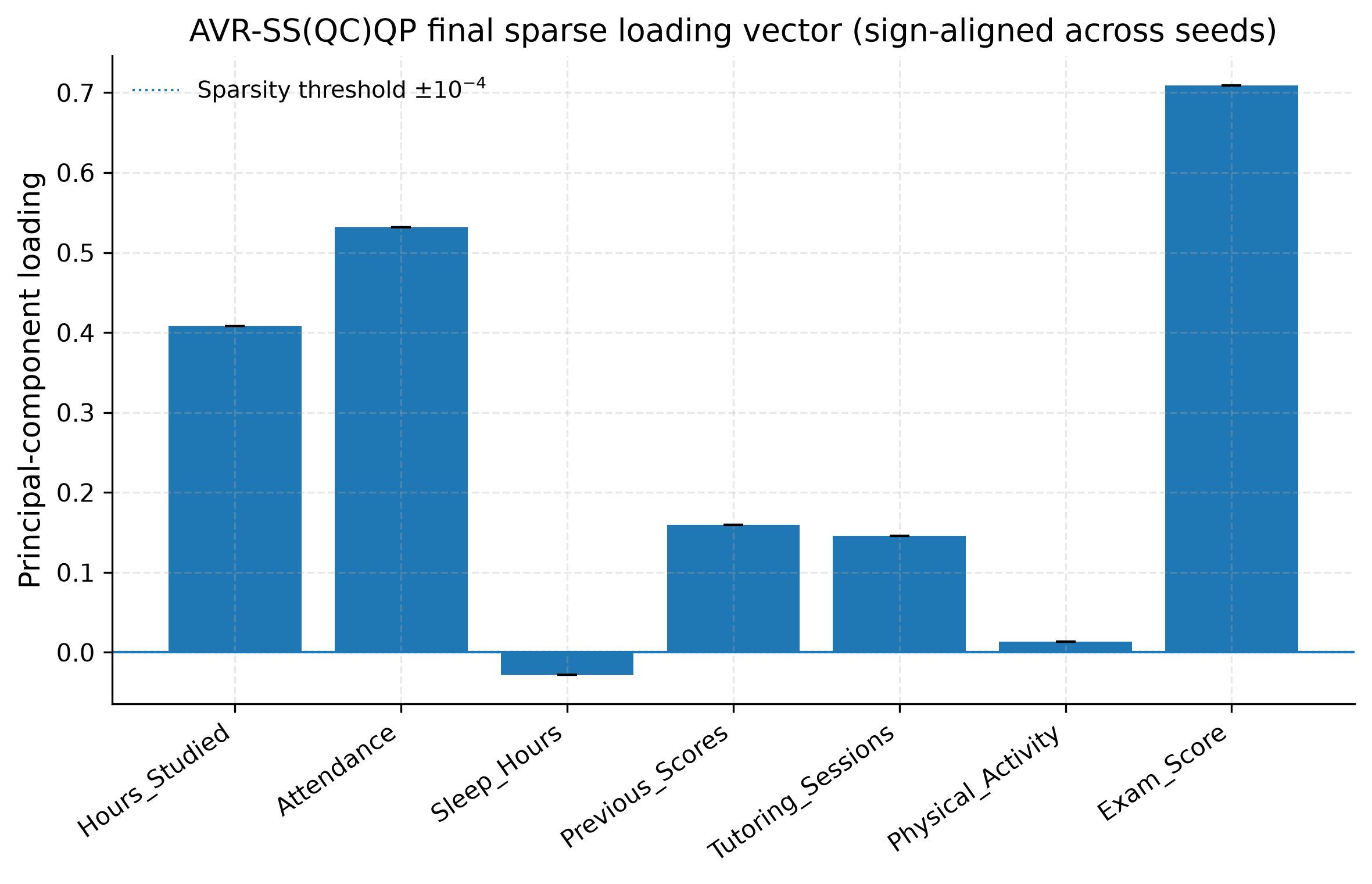}
\vspace{-2mm}

{\small (b) Final sparse principal-component loadings}
\end{minipage}

\vspace{1.5mm}

\begin{minipage}[t]{0.48\textwidth}
\centering
\includegraphics[
    width=\linewidth,
    height=0.185\textheight,
    keepaspectratio
]{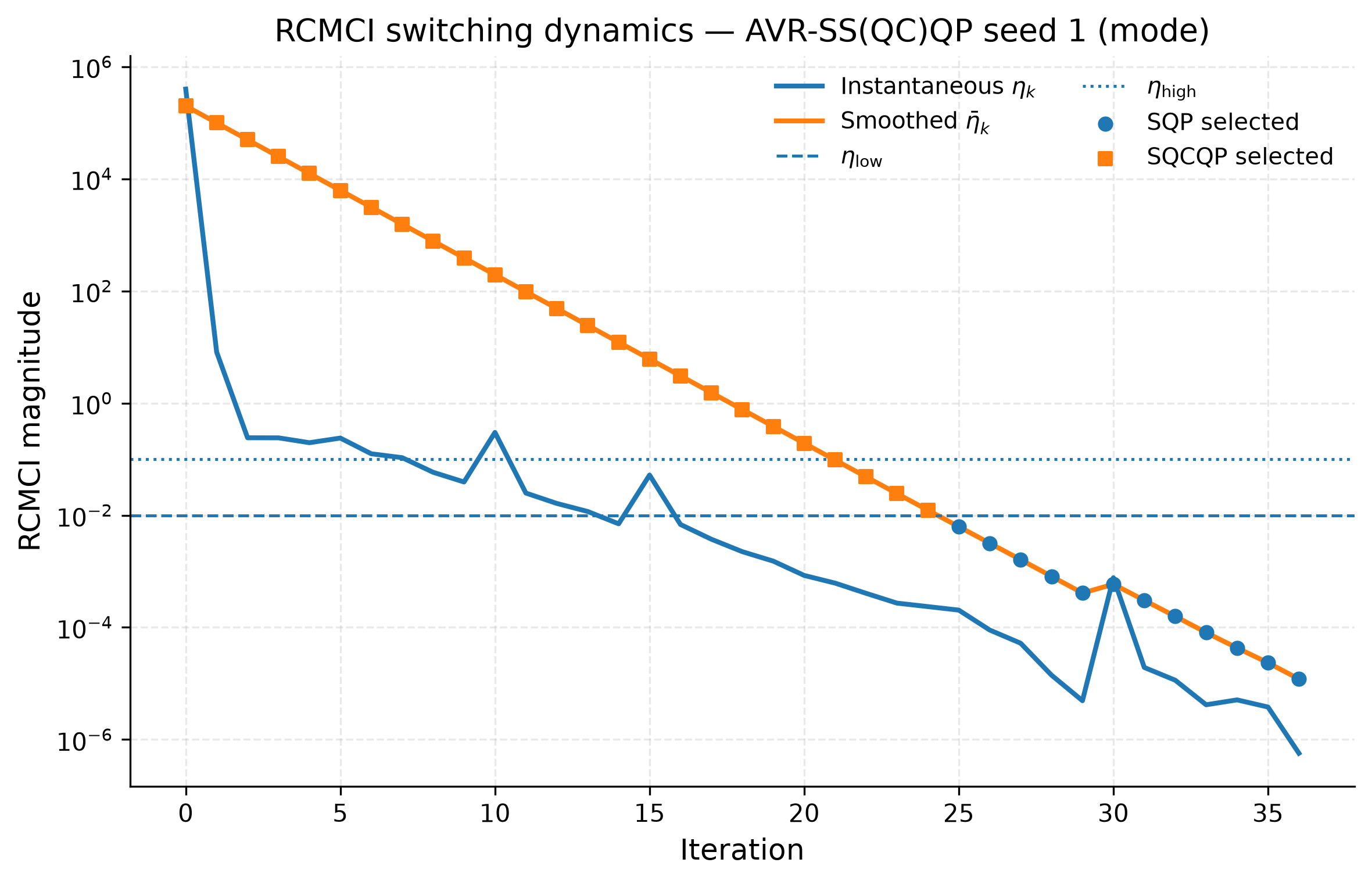}
\vspace{-2mm}

{\small (c) RCMCI switching dynamics}
\end{minipage}
\hfill
\begin{minipage}[t]{0.48\textwidth}
\centering
\includegraphics[
    width=\linewidth,
    height=0.185\textheight,
    keepaspectratio
]{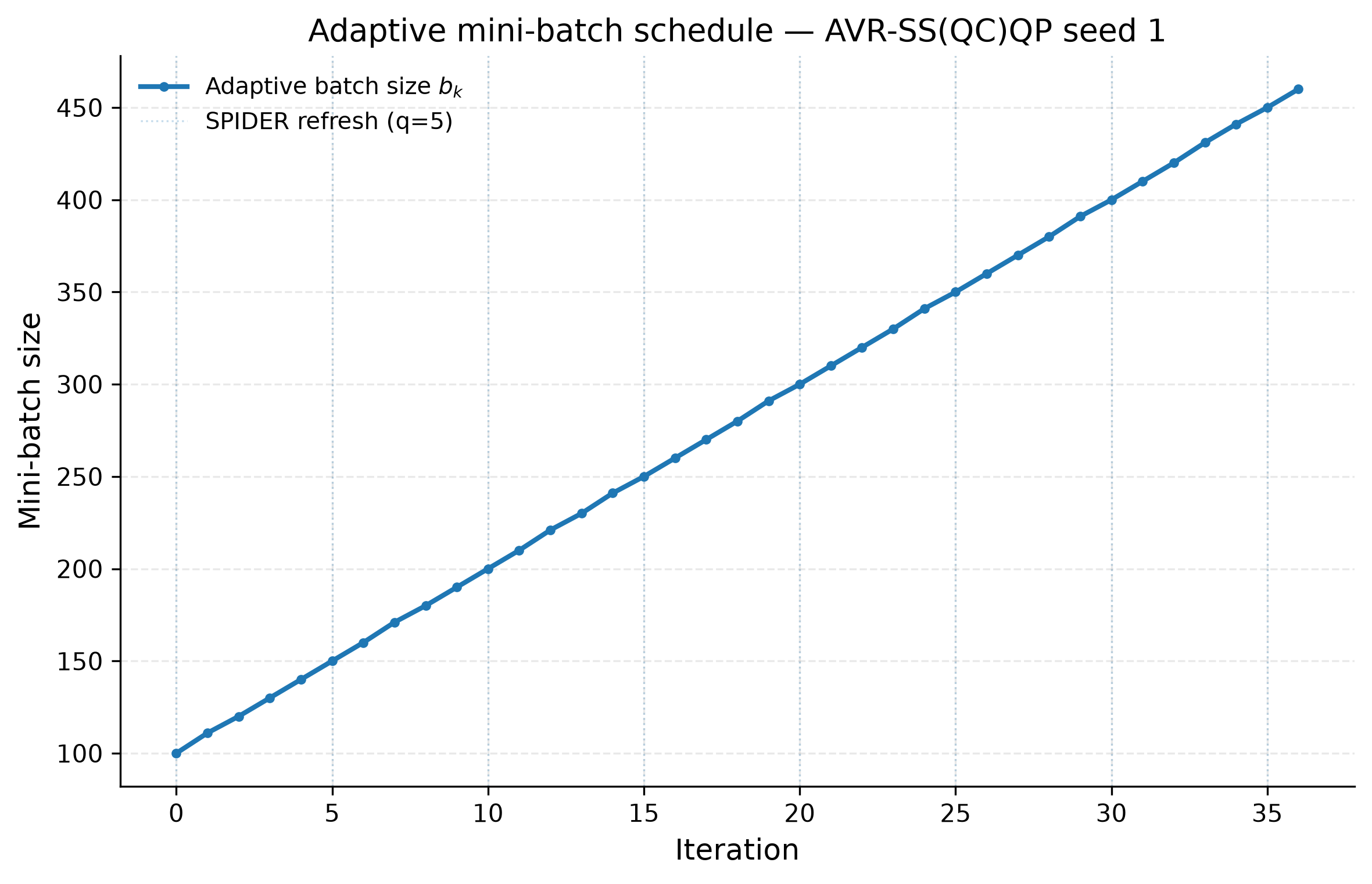}
\vspace{-2mm}

{\small (d) Adaptive mini-batch schedule}
\end{minipage}

\vspace{1.5mm}

\begin{minipage}[t]{0.48\textwidth}
\centering
\includegraphics[
    width=\linewidth,
    height=0.185\textheight,
    keepaspectratio
]{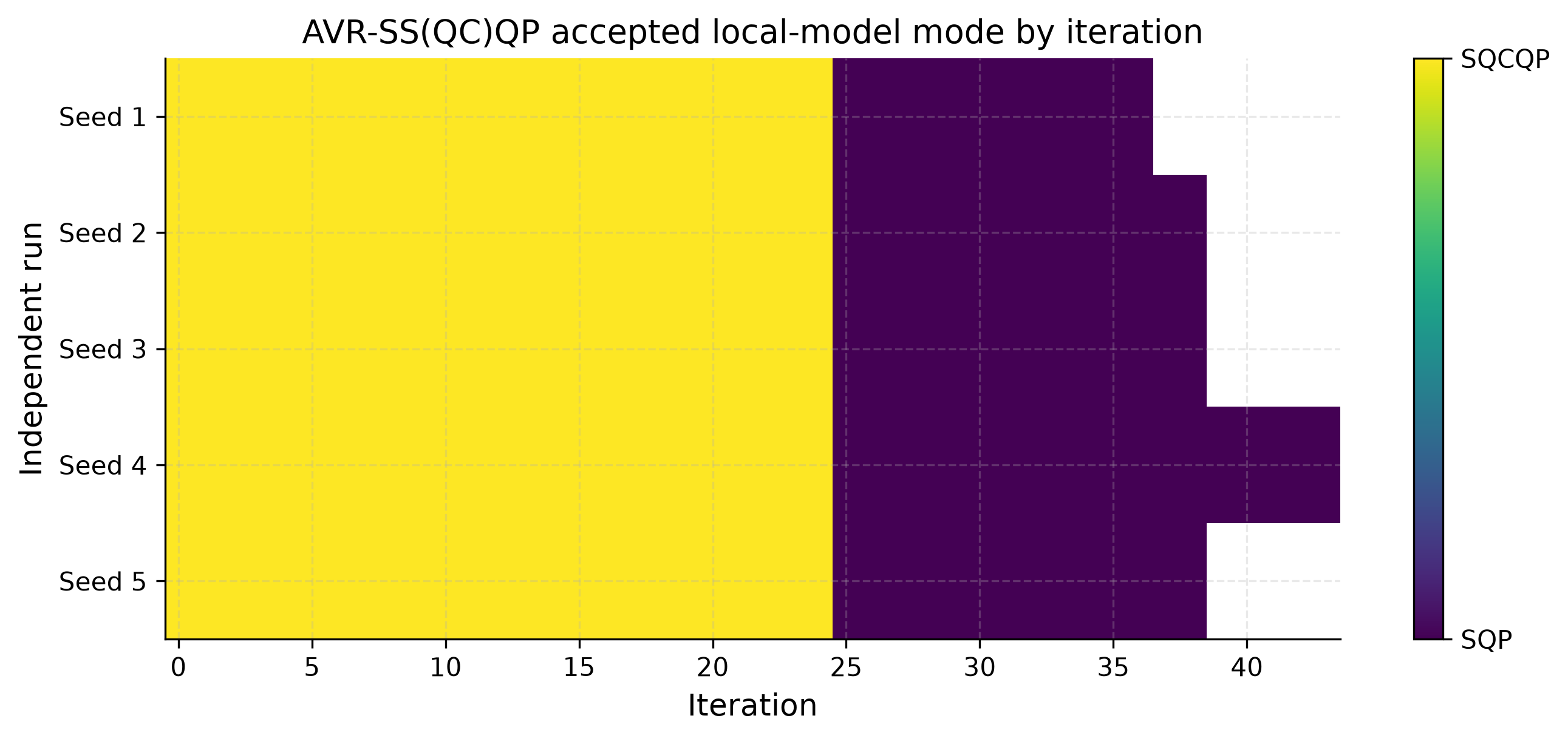}
\vspace{-2mm}

{\small (e) SQP--SQCQP mode selection across seeds}
\end{minipage}
\hfill
\begin{minipage}[t]{0.48\textwidth}
\centering
\includegraphics[
    width=\linewidth,
    height=0.185\textheight,
    keepaspectratio
]{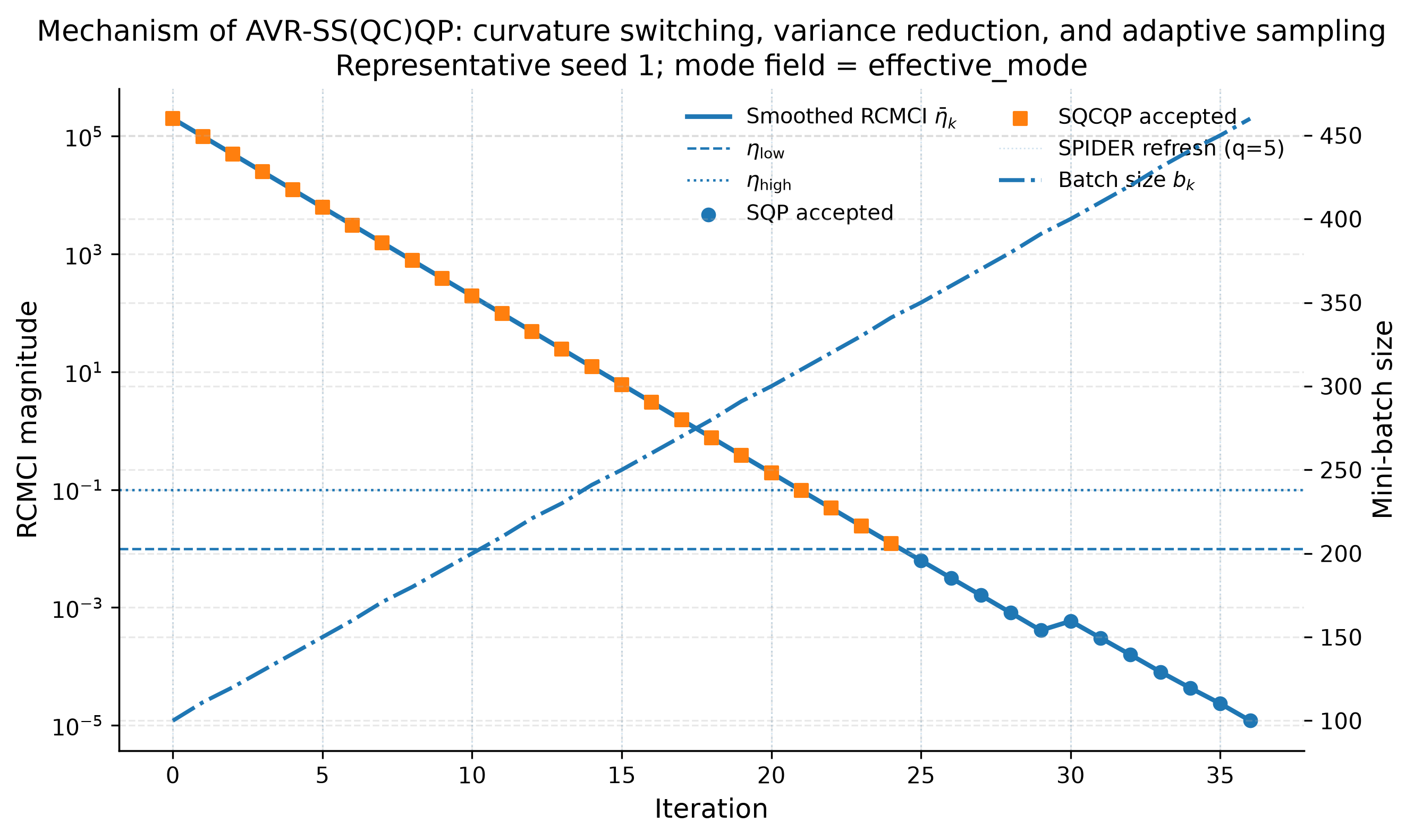}
\vspace{-2mm}

{\small (f) Combined switching and variance-reduction mechanism}
\end{minipage}

\vspace{1.5mm}

\begin{minipage}[t]{0.48\textwidth}
\centering
\includegraphics[
    width=\linewidth,
    height=0.185\textheight,
    keepaspectratio
]{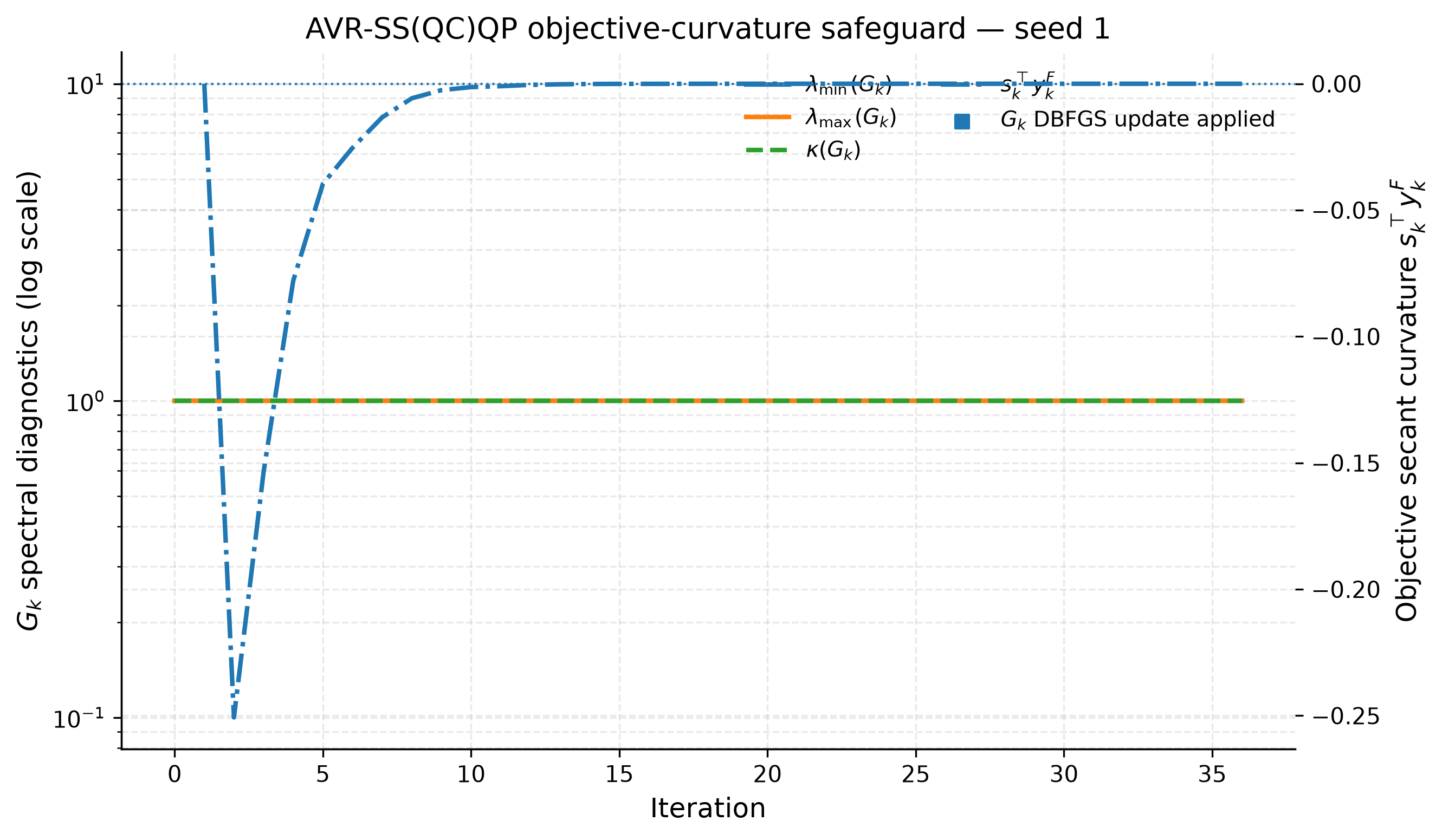}
\vspace{-2mm}

{\small (g) Objective-curvature safeguard diagnostics}
\end{minipage}
\hfill
\begin{minipage}[t]{0.48\textwidth}
\centering
\includegraphics[
    width=\linewidth,
    height=0.185\textheight,
    keepaspectratio
]{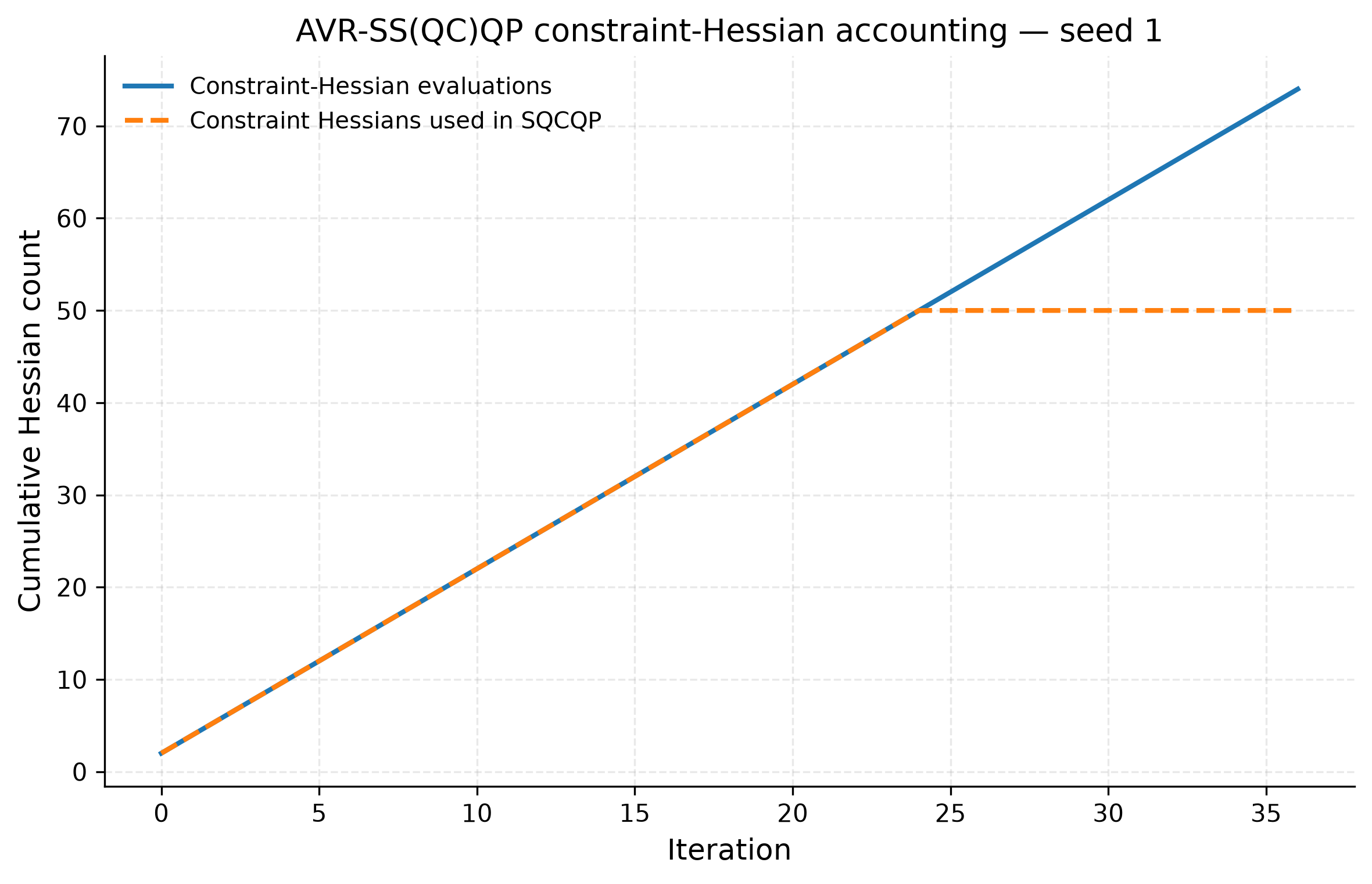}
\vspace{-2mm}

{\small (h)  Constraint-Hessian evaluation and SQCQP-use counts}
\end{minipage}

\caption{Solution-quality and internal-mechanism diagnostics for
AVR-SS(QC)QP. Panel (a) relates final objective quality to KKT
stationarity, while panel (b) displays the recovered Sparse PCA
loading vector. Panels (c)--(e) illustrate the RCMCI switching
criterion, adaptive mini-batch evolution, and SQP--SQCQP mode
selection. Panel (f) summarizes the interaction between curvature-aware
switching, SPIDER variance reduction, and adaptive sampling, whereas
panel (g) reports the objective-curvature safeguard behavior.Panel~(h) compares the cumulative constraint-Hessian evaluations with
their cumulative use in SQCQP subproblems.}
\label{fig:avr_diagnostics}
\end{figure*}
\subsection{Numerical Results and Discussion}
\label{subsec:numerical_results}
Figures~\ref{fig:benchmark_summary} and~\ref{fig:avr_diagnostics} summarize the comparative performance and internal behavior of AVR-SS(QC)QP. As reported in Table~\ref{tab:aggregate_results}, AVR-SS(QC)QP attains a mean objective of approximately \(-1.76478\), a KKT residual of \(4.65\times10^{-6}\), and negligible
constraint violation. AALM-LO reaches comparable final stationarity, whereas substantially larger KKT residuals are obtained by StoSQP-CRZ and Constrained-Adam. In particular, the small violation of Constrained-Adam together with its larger KKT residual illustrates that feasibility alone does not ensure constrained stationarity.

\begin{table*}[t]
\centering
\caption{Aggregate Sparse PCA results over five seeds, reported as
mean $\pm$ 95\% Student-$t$ confidence-interval half-width.}
\label{tab:aggregate_results}

\scriptsize
\setlength{\tabcolsep}{4pt}

\resizebox{\textwidth}{!}{%
\begin{tabular}{lccccc}
\hline
Method & Objective & KKT residual & Violation & Oracle calls & Time (s) \\
\hline

AVR-SS(QC)QP
&
$-1.76477666082 \pm 7.15\times10^{-11}$
&
$(4.65 \pm 2.48)\times10^{-6}$
&
$(3.39 \pm 4.12)\times10^{-11}$
&
$72{,}858 \pm 6{,}295$
&
$0.442 \pm 0.044$
\\

AALM-LO
&
$-1.76477666680 \pm 7.10\times10^{-9}$
&
$(4.13 \pm 4.52)\times10^{-6}$
&
$(3.42 \pm 4.03)\times10^{-9}$
&
$(2.222 \pm 0.358)\times10^{7}$
&
$15.942 \pm 2.353$
\\

StoSQP-CRZ
&
$-1.08550411 \pm 2.01\times10^{-3}$
&
$0.8045 \pm 0.0163$
&
$(3.769 \pm 0.177)\times10^{-5}$
&
$92{,}160 \pm 0$
&
$0.227 \pm 0.074$
\\

Constrained-Adam
&
$-1.756619 \pm 1.864\times10^{-3}$
&
$0.1585 \pm 0.0198$
&
$(6.49 \pm 6.18)\times10^{-7}$
&
$76{,}800 \pm 0$
&
$0.142 \pm 0.020$
\\

\hline
\end{tabular}%
}
\end{table*}

Because different amounts of stochastic-gradient information are used per iteration, oracle evaluations provide a more meaningful efficiency measure than iteration counts alone. Figures~\ref{fig:benchmark_summary}(d), (f), and~(g) therefore report progress and final cost in terms of cumulative sample-gradient evaluations. AVR-SS(QC)QP requires about \(7.29\times10^4\) oracle calls, which is of the same order as StoSQP-CRZ and Constrained-Adam, while AALM-LO requires approximately \(2.22\times10^7\). Thus, the small KKT residual of AVR-SS(QC)QP is
obtained without the substantially larger oracle cost observed for AALM-LO.

The corresponding runtimes are shown in
Figure~\ref{fig:benchmark_summary}(e). AVR-SS(QC)QP requires about
\(0.442\) seconds, compared with \(0.227\) and \(0.142\) seconds for
StoSQP-CRZ and Constrained-Adam, respectively, but remains far below
the \(15.9\) seconds required by AALM-LO. The additional cost reflects
the local SQP/SQCQP solves, switching mechanism, and globalization.
Wall-clock results should nevertheless be interpreted together with
oracle counts because they depend on implementation details.

Figure~\ref{fig:avr_diagnostics} illustrates the internal mechanisms
of AVR-SS(QC)QP. Panels~(a)--(b) show final solution quality and the
recovered Sparse PCA loading vector. Panels~(c)--(e) illustrate the
RCMCI trajectory, adaptive batching, and SQP--SQCQP switching, while
panels~(f)--(g) summarize the interaction among SPIDER refreshes,
model switching, and quasi-Newton curvature safeguards. Panel~(h)
compares cumulative constraint-Hessian evaluations with their use in
SQCQP subproblems. The eventual plateau of the SQCQP-use curve is
consistent with the transition to the SQP regime established in the
convergence analysis.

Paired two-sided Wilcoxon signed-rank tests were additionally applied
to the final objective values using runs paired by random seed. The
results are reported in Table~\ref{tab:wilcoxon_compact}.

\begin{table}[t]
\centering
\caption{Paired Wilcoxon signed-rank tests for final objective values,
with differences defined as
\(F_{\mathrm{AVR}}-F_{\mathrm{baseline}}\).}
\label{tab:wilcoxon_compact}
\small
\begin{tabular}{lccc}
\hline
Baseline & Statistic & Exact \(p\)-value & Median difference \\
\hline
AALM-LO & \(0\) & \(0.0625\) & \(1.12\times10^{-8}\) \\
StoSQP-CRZ & \(0\) & \(0.0625\) & \(-6.80\times10^{-1}\) \\
Constrained-Adam & \(0\) & \(0.0625\) & \(-6.45\times10^{-3}\) \\
\hline
\end{tabular}
\end{table}

The exact two-sided \(p\)-value is \(0.0625\) in all three comparisons,
so statistical significance at the \(5\%\) level is not obtained.
With only five paired observations, however, \(0.0625\) is the smallest
attainable exact two-sided Wilcoxon \(p\)-value in the absence of ties.
The tests are therefore interpreted only as supplementary evidence.
The near-zero median difference against AALM-LO is consistent with
their similar final objectives, whereas the negative differences
against StoSQP-CRZ and Constrained-Adam reflect the lower objective
values obtained by AVR-SS(QC)QP.

Overall, AVR-SS(QC)QP achieves a small KKT residual and negligible
constraint violation with an oracle budget comparable to the
fixed-batch stochastic baselines. The internal diagnostics further
confirm the intended interaction among variance reduction, adaptive
sampling, curvature-based SQP--SQCQP switching, and globalization.

\noindent
\section{Conclusions}\label{sec6}
This work introduced AVR-SS(QC)QP, an adaptive variance-reduced stochastic sequential (quadratic constrained) quadratic programming framework for large-scale nonlinear constrained optimization. The method combines SPIDER variance reduction and adaptive mini-batching with safeguarded positive-definite quasi-Newton models, exact \(\ell_1\)-merit-function globalization, and adaptive switching between SQP and SQCQP subproblems. The proposed RCMCI provides a curvature-aware mechanism for determining when second-order constraint information is sufficiently important to justify the additional cost of an SQCQP model, while the SQP fallback and Armijo globalization maintain a common descent mechanism for both local models. Convergence of the proposed method is justified under some mild assumptions. In addition to this, the proposed method is verified with constrained Sparse PCA and compared with StoSQP-CRZ, AALM-LO, and hard-constrained Adam. 
\section*{Declarations}
\begin{itemize}
    \item {\bf Funding:} This article is partially funded by the NBHM research project entitled ``{\it Higher Order Stochastic Methods for Large Scale Multi-objective Optimization Problems}''. Author {\it Md Abu Talhamainuddin Ansary} is thankful to NBHM for this funding.
    \item  {\it Conflict of interest/Competing interests:} The authors declare no conflict of interest.
\item {\it Ethics approval and consent to participate:} Not applicable
\item {\it Consent for publication:} Not applicable
\item {\it Data availability:} The dataset used in this study was obtained from Kaggle and is publicly available at: https://www.kaggle.com/datasets/lainguyn123/student-performance-factors.
\item {\it Materials availability:} Not applicable
\item {\it Code availability:} Codes will be made available from corresponding author based on requests.
\item {\it Author contribution:} Each author has contributed equally 
\end{itemize}

\end{document}